\documentclass[a4paper,10pt]{amsart}
\usepackage{amsmath,amsthm,amssymb,latexsym,enumerate,xcolor,hyperref}
\usepackage{graphicx}
\usepackage{esint}
\usepackage{comment}
\usepackage{eucal}
\usepackage{tikz}

\numberwithin{equation}{section}

\begin{document}

\newtheorem{thm}{Theorem}[section]
\newtheorem{prop}[thm]{Proposition}
\newtheorem{lem}[thm]{Lemma}
\newtheorem{cor}[thm]{Corollary}
\newtheorem{rem}[thm]{Remark}
\newtheorem{ex}[thm]{Example}
\newtheorem*{defn}{Definition}

\newcommand{\DD}{\mathbb{D}}
\newcommand{\NN}{\mathbb{N}}
\newcommand{\ZZ}{\mathbb{Z}}
\newcommand{\QQ}{\mathbb{Q}}
\newcommand{\RR}{\mathbb{R}}
\newcommand{\CC}{\mathbb{C}}
\renewcommand{\SS}{\mathbb{S}}

\renewcommand{\theequation}{\arabic{section}.\arabic{equation}}

\newcommand\ddfrac[2]{\frac{\displaystyle #1}{\displaystyle #2}}

\newcommand{\supp}{\mathop{\mathrm{supp}}}    

\newcommand{\cb}{\color{blue}}   
\newcommand{\cre}{\color{red}}   
\newcommand{\re}{\mathop{\mathrm{Re}}}   
\newcommand{\im}{\mathop{\mathrm{Im}}}   
\newcommand{\dist}{\mathop{\mathrm{dist}}}  
\newcommand{\link}{\mathop{\circ\kern-.35em -}}
\newcommand{\spn}{\mathop{\mathrm{span}}}   
\newcommand{\ind}{\mathop{\mathrm{ind}}}   
\newcommand{\rank}{\mathop{\mathrm{rank}}}   
\newcommand{\Fix}{\mathop{\mathrm{Fix}}}   
\newcommand{\codim}{\mathop{\mathrm{codim}}}   
\newcommand{\conv}{\mathop{\mathrm{conv}}}   
\newcommand{\osc}{\mathop{\mathrm{osc}}}  
\newcommand{\epsi}{\mbox{$\varepsilon$}}
\newcommand{\eps}{\mathchoice{\epsi}{\epsi}
{\mbox{\scriptsize\epsi}}{\mbox{\tiny\epsi}}}
\newcommand{\cl}{\overline}
\newcommand{\pa}{\partial}
\newcommand{\ve}{\varepsilon}
\newcommand{\zi}{\zeta}
\newcommand{\Si}{\Sigma}

\newcommand{\cC}{\mathcal{C}}
\newcommand{\cD}{{\mathcal D}}
\newcommand{\cE}{{\mathcal E}}
\newcommand{\cF}{{\mathcal F}}
\newcommand{\cG}{{\mathcal G}}
\newcommand{\cH}{{\mathcal H}}
\newcommand{\cI}{{\mathcal I}}
\newcommand{\cJ}{{\mathcal J}}
\newcommand{\cK}{{\mathcal K}}
\newcommand{\cL}{{\mathcal L}}
\newcommand{\cM}{\mathcal{M}}
\newcommand{\cN}{{\mathcal N}}
\newcommand{\cP}{\mathcal{P}}
\newcommand{\cR}{{\mathcal R}}
\newcommand{\cS}{{\mathcal S}}
\newcommand{\cT}{{\mathcal T}}
\newcommand{\cU}{{\mathcal U}}
\newcommand{\cV}{\mathcal{V}}
\newcommand{\cX}{\mathcal{X}}

\newcommand{\cOS}{\overset{o}{\cS}}

\newcommand{\B}{\bullet}
\newcommand{\ol}{\overline}
\newcommand{\ul}{\underline}
\newcommand{\vp}{\varphi}
\newcommand{\AC}{\mathop{\mathrm{AC}}}   
\newcommand{\Lip}{\mathop{\mathrm{Lip}}}   
\newcommand{\es}{\mathop{\mathrm{esssup}}}   
\newcommand{\les}{\mathop{\mathrm{les}}}   
\newcommand{\nid}{\noindent}
\newcommand{\pzr}{\phi^0_R}
\newcommand{\pir}{\phi^\infty_R}
\newcommand{\psr}{\phi^*_R}
\newcommand{\pow}{\frac{N}{N-1}}
\newcommand{\ncl}{\mathop{\mathrm{nc-lim}}}   
\newcommand{\nvl}{\mathop{\mathrm{nv-lim}}}  
\newcommand{\la}{\lambda}
\newcommand{\La}{\Lambda}    
\newcommand{\de}{\delta}    
\newcommand{\fhi}{\varphi} 
\newcommand{\ga}{\gamma}    
\newcommand{\ka}{\kappa}   
\newcommand{\Te}{\Theta}   

\newcommand{\core}{\heartsuit}
\newcommand{\diam}{\mathrm{diam}}
\newcommand{\loc} {\mathrm{loc}}
\newcommand{\reg} {\mathrm{reg}}
\newcommand{\sing} {\mathrm{sing}}

\newcommand{\lan}{\langle}
\newcommand{\ran}{\rangle}
\newcommand{\tr}{\mathop{\mathrm{tr}}}
\newcommand{\diag}{\mathop{\mathrm{diag}}}
\newcommand{\dv}{\mathop{\mathrm{div}}}
\newcommand{\one}{\mathcal{X}}

\newcommand{\al}{\alpha}
\newcommand{\be}{\beta}
\newcommand{\Om}{\Omega}
\newcommand{\na}{\nabla}

\newcommand{\nr}{\Vert}
\newcommand{\De}{\Delta}
\newcommand{\om}{\omega}
\newcommand{\si}{\sigma}
\newcommand{\te}{\theta}
\newcommand{\Ga}{\Gamma}


\title[Short-time with non-constant boundary values]{Short-time behavior and invariant surfaces \\ for caloric functions with \\
non-constant boundary values}

\author[D. Berti]{Diego Berti}
\address{Dipartimento di Matematica ``G. Peano'',
Universit\` a di Torino, Torino (Italy).}
    \email{diego.berti@unito.it}

\author[R. Magnanini]{Rolando Magnanini}
\address{Dipartimento di Matematica ed Informatica ``U.~Dini'',
Universit\` a di Firenze, Firenze (Italy).}
    \email{rolando.magnanini@unifi.it}
    \urladdr{https://people.dimai.unifi.it/magnanini/}

\author[M. Marini]{Michele Marini}
\address{Dipartimento di Ingegneria (DING), Universit\`a del Sannio, Benevento (Italy).}
    \email{mmarini@unisannio.it}


\begin{abstract}
{
We consider a Cauchy-Dirichlet problem for a heat equation with variable coefficients 
in non-divergence form. The initial values are assumed to be homogeneous, while the Dirichlet boundary values are non-constant. In this setting, we derive an asymptotic formula for the short-time behavior of the solution, which extends the celebrated Varadhan formula. We stress the fact that the boundary values are allowed to vanish or change sign. 
\par
The formula is obtained by combining the original ideas of Varadhan with those of Evans and Ishii, pertaining to the theory of viscosity solutions, and some further remarks. In passing, we prove the elliptic counterpart of the formula. This concerns the slow-diffusion behavior of the {solutions} of the resolvent equation. 
\par
Furthermore, for the case of the classical Laplace operator, we prove asymptotic formulas for the heat content and the mean value of the solution of the resolvent equation on spheres touching the boundary of the domain. These extend to the case of non-constant boundary values, certain formulas previously obtained by the second author and S. Sakaguchi. The formulas involve the principal curvatures at the touching point and the Dirichlet boundary values.
\par
We then use our formulas to describe time-invariant surfaces for the heat equation, in presence of non-constant Dirichlet boundary values. We give two
symmetry results in case of non-negative Dirichlet boundary values and a non-existence result, when those values change sign.  
}
\end{abstract}


\maketitle



\section{Introduction}

In this paper, we 
prove some new asymptotic results related to the short-time behavior of solutions of certain parabolic equations. These results extend two classical formulas, often named \textit{Varadhan's formulas} \cite{Va}, and a formula, obtained in \cite{MS1} by S. Sakaguchi and the second author, which has to do with the heat content of spheres. These formulas will be used to  obtain rigidity results of the so-called time-invariant surfaces for certain parabolic Cauchy-Dirichlet problems.

\subsection{Varadhan's formulas}

The \textit{Varadhan formulas} are now a classical result (see \cite{Va}). The first formula gives the asymptotic behavior for small times of the heat kernel $\Phi(x,y,t)$ of the quite general parabolic equation 
\begin{equation*}
u_t-\tr \bigl[A(x)\,\na^2 u\bigr]=0 \ \mbox{ in } \ \RR^N\times(0,\infty).
\end{equation*}
Here and in the sequel, $N\ge 2$, the symmetric $N\times N$ matrix of functions $A(x)=\{ A_{ij}(x)\}_{i, j=1,\dots, N}$ satisfies a uniform ellipticity condition, $\na^2 u$ denotes the Hessian matrix of $u$, and $\tr$ is the usual trace operator acting on matrices.
\par
The second formula concerns the slow-diffusion behavior in a domain $\Om$ of the solutions $U_\ve=U_\ve(x)$ of the $1$-parameter family of Dirichlet problems
\begin{equation*}
-\ve^2\tr \bigl[A(x)\,\na^2 U_\ve\bigr] +U_\ve=0 \ \mbox{ in } \ \Om, \quad 
U_\ve=1 \ \mbox{ on} \ \pa\Om,
\end{equation*}
for $\ve>0$.
Here, $\Om$ is a bounded domain in $\RR^N$, $N\ge 2$, with boundary $\pa\Om$. The precise structure and regularity of $A(x)$ and $\Om$ will be listed in Section \ref{sec:assumptions}.\footnote{Note that, for the purposes of this paper, we prefer to adopt slightly different notations from those used in \cite{Va}.}
\par
The two asymptotic formulas relate the specified solutions to some relevant distance functions. In fact, one obtains the formulas:
\begin{equation*}
\lim_{t\to 0^+} \left[-4t \log \Phi(x,y,t)\right]=d(x,y)^2, \ x, y\in\RR^N,
\end{equation*}
and 
\begin{equation}
\label{varadhan-elliptic}
\lim_{\ve\to 0^+} \bigl[-\ve \log U_\ve(x)\bigr]=d_{\pa\Om}(x), \ x\in\ol{\Om}.
\end{equation}
Here,  $d(x,y)$ is the length of the shortest path joining $x$ to $y$ in a Riemannian metric derived from the matrix $A(x)$ and $d_{\pa\Om}(x)$ is the minimum of $d(x,y)$ for $y\in\pa\Om$. (The precise definitions for these distances are listed in Section \ref{sec:assumptions}.) The formula for the heat kernel has been recently used in \cite{HL}, in connection with Gel'fand's inverse problem for Riemannian manifolds. 
\par
Also, with a little more effort, from the results in \cite{Va}, it is not difficult to deduce a formula for the solution $u=u(x,t)$ of the Cauchy-Dirichlet problem:
\begin{eqnarray*}
&&u_t-\tr \bigl[A(x)\,\na^2 u\bigr]=0 \ \mbox{ in } \ \Om\times (0,\infty),  \nonumber \\
&&u=0 \ \mbox{ on } \ \Om\times \{ 0\}, \\
&&u=1 \ \mbox{ on } \ \pa\Om\times (0,\infty).\nonumber
\end{eqnarray*}
In fact, $u$ and $U_\ve$ are related by the modified Laplace transformation:
\begin{equation}
\label{laplace}
U_\ve(x)=\ve^{-2} \int_0^\infty e^{-t/\ve^2} u(x,t)\,dt, \ \ \ve>0.
\end{equation}
The asymptotic formula for $u$ thus reads as
\begin{equation}
\label{varadhan}
\lim_{t\to 0^+} \bigl[-4t \log u(x,t)\bigr]=d_{\pa\Om}(x)^2, \ x\in\ol{\Om},
\end{equation}
and can be derived from the formula for $U_\ve$, via \cite[Theorem 4.7]{Va} (which we recall in Lemma \ref{lem:laplace-transform} below).
\par
Besides Varadhan's original proofs in \cite{Va}, based on the construction of suitable barriers, alternative proofs of those formulas were obtained by Evans and Ishii in \cite{EI}, by means of the theory of viscosity solutions and in relation with the works of
Fleming \cite{Fl1, Fl2} on optimal stochastic control. Applications to random perturbations of dynamical systems can be found in the book \cite{FW} of Freidlin and Wentzell. The interested reader is also referred to the works \cite{MS2,MS4}, \cite{BM1,BM2,BM3}, for extensions to various nonlinear regimes, involving the porous medium equation, the evolution $p$-Laplace equation, certain fast-diffusion equations, the game-theoretic $p$-Laplace equations, or the Pucci operators. In \cite{BM1,BM2,BM3}, in the relevant settings, it is also specified the rates of convergence, as $O(\ve \log \ve)$ for $\ve\to 0^+$ or $O(t \log t)$ for $t\to 0^+$.

\subsection{Time-invariant and other invariant surfaces}
Some of the mentioned results were instrumental to prove, in the relevant settings, certain rigidity results for the so-called time-invariant surfaces. If $u=u(x,t)$ is a solution of some evolution equation in a domain $\Om$, a level surface $\Ga\subset\Om$ of $u$ is called \textit{time-invariant} if 
\begin{equation}
\label{time-invariant}
u(x,t)=c(t) \ \mbox{ for } \ (x,t)\in \Ga\times (a,b),
\end{equation}
for some function $c:(a,b)\to\RR$ of the time $t$. It is worthwhile to notice in passing that, thanks to \eqref{laplace}, a time-invariant surface for $u$ is also, so as to speak, a \textit{diffusion-invariant} surface for $U_\ve$, when $a=0$ and $b=\infty$. 
\par
Time-invariance introduces a strong {overdetermination} on the relevant solution. We refer the reader to \cite{MM2} for a (local) characterization of solutions of linear and nonlinear parabolic equations (even of Finsler type), whose level surfaces are \textit{all} time-invariat --- the so-called \textit{Matzoh Ball Soup Problem} introduced in \cite{Kl}, \cite{Za}. 
\par
When we consider certain Cauchy-Dirichlet problems,
a single time-invariant surface 
is generally proved to have some symmetry. In the relevant metrics, {it has been shown} to be
a sphere (in bounded regimes), or  a hyperplane, an infinite cylinder, or even a  right helicoid (in unbounded regimes). We refer the interested reader to a long list of papers on this issue: \cite{MPrS,MPeS}, \cite{MS1}-\cite{MS5}, \cite{MM1}, which entails linear and nonlinear parabolic equations, even in anisotropic frameworks.

\par
One of the key points of the strategy to prove such rigidity results is to observe that, if a Varadhan-type formula is in force for the relevant boundary value problem, then a time-invariant surface is also a level surface of some distance function, generally the distance from the boundary $d_{\pa\Om}$ (see Section \ref{sec:assumptions}).   In this case $\Ga$ would be \textit{parallel} to $\pa\Om$. This information is sometimes sufficient to prove spherical symmetry in bounded regimes, by means of Serrin's \textit{method of moving planes} (\cite{MS4}).  
Alternatively, another method of proving symmetry is to show that the principal curvatures of time-invariant surfaces satisfy some constraint. In fact, one proves that they have a \textit{constant mean curvature} or, more generally, they are \textit{Wirtinger's surfaces}. {Symmetry} is then gained by applying \textit{Alexandrov's reflection principle} (\cite{MS1}), or yet \textit{$P$-function} techniques, for bounded regimes, or the \textit{sliding method}, for unbounded regimes (\cite{MS3,MS5}). 
\par

\subsection{Non-constant initial or boundary values}
It is natural to consider 
a more general Cauchy-Dirichlet problem, in which the initial and Dirichlet boundary values are allowed to be non-constant. In other words, we examine the problem:
\begin{eqnarray}
\label{cauchy-dirichlet}
&&u_t-\tr \bigl[A(x)\,\na^2 u\bigr]=0 \ \mbox{ in } \ \Om\times (0,\infty),  \nonumber \\
&&u=u_0 \ \mbox{ on } \ \Om\times \{ 0\}, \\
&&u=f \ \mbox{ on } \ \pa\Om\times (0,\infty), \nonumber
\end{eqnarray}
where $u_0\in L^2(\Om)$ and $f\in C(\pa\Om)$ are given non-constant functions.
\par
It is not difficult to check that the functions $U_\ve$ defined by \eqref{laplace}
satisfy the Dirichlet problem:
\begin{equation}
\label{resolvent-dirichlet}
\ve^2\tr \bigl[A(x)\,\na^2 U_\ve\bigr] -U_\ve=u_0 \ \mbox{ in } \ \Om, \quad 
U_\ve=f \ \mbox{ on } \ \pa\Om, 
\end{equation}
for fixed $\ve>0$.
Also in this case, time-invariant and diffusion-invariant surfaces coincide.
\par
The slow diffusion behavior for problem \eqref{resolvent-dirichlet} has been recently considered in a couple of interesting papers related to applications to image processing, computer vision, and manufacturability,  (\cite{HMOSY}, \cite{HMTY}). In fact,  for the case in which the relevant elliptic operator is the classical Laplacian (i.e. $A(x)=I$), in \cite{HMOSY} it is assumed that $u_0$ and $f$ are non-negative and $u_0$ is allowed to be zero only on an open subset $E$ with $\ol{E}\subset\Om$ (and finitely many connected components). Under suitable regularity assumptions on $E$ and $f$ (see \cite[Section 1.2]{HMOSY}), it is then proved that
\begin{equation*}
\lim_{\ve\to 0^+} \bigl[-\ve \log U_\ve(x)\bigr]=d_{\pa E}(x) \ \mbox{ for } \ x\in\ol{E},
\end{equation*}
and the limit is uniform. 
Moreover, we notice here that, for the solution $u$ of \eqref{cauchy-dirichlet}, it is also possible to use Lemma \ref{lem:laplace-transform} and deduce a formula analogous to \eqref{varadhan}:
\begin{equation*}
\lim_{t\to 0^+} \bigl[-4t \log u(x,t)\bigr]=d_{\pa E}(x)^2, \ x\in\ol{E}.
\end{equation*}
It is interesting to point out that these formulas do not depend on the boundary values $f$, but only on the set of positivity of $u_0$.
Also note, in passing, that the last formula informs us that a possible time-invariant surface contained in $E$ must be parallel to $\pa E$.

If the set $E$ extends to the boundary, then the Dirichlet data comes into play, as is clear from the classical Varadhan's asymptotics. In this paper, we shall address the case in which $u_0\equiv 0$, i.e. the core set $E$ is exactly $\Om$, and $f$ is non-constant and is allowed to vanish or even change sign on $\pa\Om$. We shall treat this case for a general matrix of coefficients $A(x)$.
\par
It should be noticed that the extension of Varadhan's formulas is straightforward when we suppose that $a_-\le f\le a^+$ on $\pa\Om$, for some positive constants $a^-, a^+$. This is the content of Proposition \ref{prop:varadhan-easy-extension} and its proof  is an adaptation of an argument contained in \cite[Theorem 3.7]{MS4}. Instead, the situation becomes more intricate when $f$ is allowed to vanish or change sign. The discussion of {these cases} is one of the highlights of this paper  and will be discussed in Section \ref{sec:non-constant}.
\par
A first result is encoded in the following theorem and concerns the case of nonnegative boundary values.

\begin{thm}[Extended Varadhan's formulas]
\label{thm:varadhan-elliptic-parabolic}
{Let $\Om \subset \mathbb R^N$ be a bounded domain} with boundary $\pa\Om$ of class $C^{1,1}$.
Choose $u_0\equiv 0$, let $f\in C(\pa\Om)$ be non-negative (not identically zero), and set 
$$
\cS=\{x\in\pa\Om: f(x)>0\}.
$$
\par
Let $u=u(x,t)$ and $U_\ve=U_\ve(x)$ be the solutions of the problems \eqref{cauchy-dirichlet}  and \eqref{resolvent-dirichlet}.
Then, it holds that
\begin{equation}
\label{varadhan-non-negative}
\lim_{t\to 0^+} \bigl[-4t \log u(x,t)\bigr]=d_{\cS}(x)^2 
\end{equation}
and
\begin{equation}
\label{varadhan-non-negative-elliptic}
\lim_{\ve\to 0^+} \bigl[-\ve \log U_\ve(x)\bigr]=d_{\cS}(x),
\end{equation}
for $x\in \Om\cup\cS$.
Moreover, the limits hold uniformly on compact subsets of $\Om\cup\cS$.
\end{thm}
Here, $d_\cS$ is a suitable distance of points $x\in\Om$ to the positivity set $\cS$ of $f$. For its precise definition of $d_\cS$, we refer the reader to Section \ref{subsec:distance}. 
It is interesting to observe that, when $f$ is allowed to vanish, but $\ol{\cS}=\pa\Om$,  then \eqref{varadhan-non-negative} and \eqref{varadhan-non-negative-elliptic} still read as the original Varadhan's formulas \eqref{varadhan} and \eqref{varadhan-elliptic}, except at the points at which $f=0$. In fact, at those points, $-4t \log u(x,t)$ and $-\ve \log U_\ve(x)$ equal $+\infty$, for every positive $t$ and $\ve$.
\par
The proof of Theorem \ref{thm:varadhan-elliptic-parabolic} relies on a combination of the techniques used in \cite{Va}, \cite{EI}, and an additional remark. First, for $U_\ve$, we settle the case in which $f$ is the characteristic function of a relatively open subset of $\pa\Om$. For this case, we follow the framework of viscosity solutions contained in \cite{EI}. Secondly, by Theorem \ref{thm:varadhan-any-positive}, we extend the result to a general non-negative continuous function. Thirdly, we use Lemma \ref{lem:laplace-transform} below to obtain the desired result in the parabolic case.

As a straightforward corollary of Theorem \ref{thm:varadhan-elliptic-parabolic}, we can also partially treat the case of sign-changing boundary values. 

\begin{cor}[Sign-changing data]
\label{cor:sign-changing}
Take $f\in C(\pa\Om)$. {Set $u_0 \equiv 0$ and let} $\cS^+$, $\cS^-$, and $\cS$ be the {(non-empty)} sets defined by
\begin{equation}
\label{def-S12}
\cS^+=\{ x\in\pa\Om : f(x)>0\}, \quad \cS^-=\{ x\in\pa\Om : f(x)<0\}, \quad \cS=\cS^+\cup\cS^-.
\end{equation}
\par
Let $u=u(x,t)$ and $U_\ve=U_\ve(x)$ be the solutions of the problems \eqref{cauchy-dirichlet}  and \eqref{resolvent-dirichlet}.
Then, for $x\in \Om\cup\cS$, it holds that
$$
\lim_{t\to 0^+} \bigl[-4t \log |u(x,t)|\bigr]=d_{\cS}(x)^2= 
\begin{cases}
d_{\cS^+}(x)^2 \ &\mbox{ for } \ d_{\cS^+}(x)<d_{\cS^-}(x), \\
d_{\cS^-}(x)^2 \ &\mbox{ for } \ d_{\cS^+}(x)>d_{\cS^-}(x),
\end{cases} 
$$
and
$$
\lim_{\ve\to 0^+} \bigl[-\ve \log |U_\ve(x)|\bigr]= \\ d_{\cS}(x)= 
\begin{cases}
d_{\cS^+}(x) \ &\mbox{ for } \ d_{\cS^+}(x)<d_{\cS^-}(x), \\
d_{\cS^-}(x) \ &\mbox{ for } \ d_{\cS^+}(x)>d_{\cS^-}(x).
\end{cases}
$$
The limits are uniform on the relevant compact subsets of $\Om\cup\cS$. 
\end{cor}

The case in which $d_{\cS^+}(x)=d_{\cS^-}(x)$ is more delicate and, besides the sets $\cS^+$ and $\cS^-$, also the actual values of $f$ come into play. We refer the reader to Example \ref{ex:fat-G} and also Remark \ref{rem:threshold}, for a brief discussion.

\subsection{On the heat content of spheres
}
When $A(x)=I$ and $f$ is constant, some of the rigidity results for time-invariant surfaces of caloric functions 
relied on asymptotic formulas relating the principal curvatures of $\pa\Om$ to the heat content of (or the mean value of $U_\ve$ on) spheres or balls touching the boundary. (See \cite[Theorem 2.3]{MS1} and \cite[Theorem 1.3]{MS2}.)
\par
For the purposes of Section \ref{sec:invariant-surfaces}, discussed in the next subsection, when $f$ is not constant, those formulas need to be updated. In what follows, the functions $\ka_1,\dots, \ka_{N-1}: \pa\Om\to\RR$ will denote the {principal curvatures} of $\pa\Om$, with respect to the interior normal $\nu$ at points in $\pa\Om$.

\begin{thm}[Asymptotics for the heat content]
\label{thm:mean-value-asymptotics}
Let {$\Om\subset \mathbb R^N$} be a bounded domain and fix $q\in\pa\Om$. Suppose that $\pa\Om$ is of class $C^2$ in a neighborhood of $q$. Let $B_R(p)\subset\Om$ be a ball such that
$\pa B_R(p)\cap\pa\Om=\{q\}$ and  
suppose that $\ka_j(q)<1/R$, for $j=1,\dots, N-1$.
\par
Let $f \in C(\pa\Om)$ be non-negative 
and let $u=u(x,t)$ and $U_\ve=U_\ve(x)$ be the solutions of the problems
\begin{equation}
\label{heat-cauchy-dirichlet}
u_t=\De u \ \mbox{ in } \ \Om\times(0,\infty), \ \ u=0  \ \mbox{ on } \ \Om\times\{ 0\}, \ \
 u=f \ \mbox{ on } \ \pa\Om\times(0,\infty),
\end{equation}
and
\begin{equation}
\label{heat-resolvent-dirichlet}
\ve^2 \De U_\ve-U_\ve=0 \ \mbox{ in } \ \Om,  \quad
 U_\ve=f \ \mbox{ on } \ \pa\Om.
\end{equation}
\par
Then, the following formulas hold:
\begin{equation}
\label{MS-f-non-constant-heat}
\lim_{\ve\to 0^+}(4\pi^2 t)^{-\frac{N-1}{4}}\int_{\pa B_R(p)} u(x,t)\,dS_x=\frac{f(q)/\Ga\left(\frac{N+3}{4}\right)}{ \sqrt{\prod\limits_{i=1}^{N-1}\bigl[1/R-\ka_i(q)\bigr]}};
\end{equation}
\begin{equation}
\label{MS-f-non-constant}
\lim_{\ve\to 0^+}(2\pi \ve)^{-\frac{N-1}{2}}\int_{\pa B_R(p)} U_\ve(x)\,dS_x=\frac{f(q)}{ \sqrt{\prod\limits_{i=1}^{N-1}\bigl[1/R-\ka_i(q)\bigr]}}.
\end{equation}
\end{thm}
This theorem will be proved in Section \ref{subsec:heat-content}. It is clear that, when $f=1$, \eqref{MS-f-non-constant} reduces to that in \cite[Theorem 2.3]{MS1} (see also \cite[Theorem 1.3]{MS2}).

\subsection{Rigidity results for time-invariant surfaces}
We will conclude our investigations with Section \ref{sec:invariant-surfaces}.
There, we analyse how the presence of non-constant Dirichlet boundary values affects some of the rigidity results already obtained in the literature for the constant regime. 
\par 
In fact, in Section \ref{sec:invariant-surfaces}, 
we shall consider the solution 
$u=u(x,t)$ of the general Cauchy-Dirichlet problem \eqref{cauchy-dirichlet} for the heat equation (i.e. when $A(x)=I$):
\begin{equation}
\label{cauchy-dirichlet-gen}
u_t=\De u \ \mbox{ in } \ \Om\times(0,\infty), \ u=u_0  \ \mbox{ on } \ \Om\times\{ 0\}, \
 u=f \ \mbox{ on } \ \pa\Om\times(0,\infty),
\end{equation}
where $u_0\in L^2(\Om)$ and $f\in C(\pa\Om)$.
The following theorem, which extends to the new settings the symmetry results previously obtained in \cite{MS1, MS2},  tells us that the presence of a time-invariant surface \textit{for all times} is a quite strong requirement for a solution of the heat equation.

\begin{thm}[Symmetry for general boundary values]
\label{thm:rigidity-definitive}
 Let {$\Om\subset \mathbb R^N$} be a bounded domain satisfying the exterior sphere condition.
Suppose that $f \in C(\pa\Om)$, $f\not\equiv 0$,
and let $u$ be the solution of \eqref{heat-cauchy-dirichlet}.
\par
Let $D\subset\Om$ be a domain such that $\pa D$ satisfies the interior cone condition and either $\ol{D}\subset\Om$ or $\pa D$ is connected and
$\pa D\cap\Om\ne\varnothing$. If, for $\Ga=\pa D\cap\Om$, the condition
\begin{equation*}
u(x,t)=c(t) \ \mbox{ for } \ (x,t)\in \Ga\times (0,\infty)
\end{equation*}
holds for some function $c:(0,\infty)\to \RR$, then 
$D$ and $\Om$ must be concentric balls.
\end{thm}

The proof of this theorem is given in Section \ref{subsec:short-times}. It essentially relies on the study of the large-time behavior of the solution, which, in presence of the relevant invariant surface, forces the data $f$ to be constant on $\pa\Om$, by {analyticity} in space. Thus, \cite[Theorem 1.1]{MS1} applies.  
\par
It should be noticed that the proof \cite[Theorem 1.1]{MS1} only relies on the short-time behavior. In fact, when $f$ is constant, the large times give no useful information. On the other hand, the {analyticity} in time of the caloric functions tells us that assuming that a surface is invariant for short times is equivalent to assume it invariant for \textit{all} times. However,
the importance of short-times was confirmed in \cite{MS2, MS4}, in which the relevant symmetry results were obtained for certain nonlinear settings, in which analyticity can no longer be used.  
\par
For these reasons, it is interesting to analyse the case in which a surface $\Ga$ is assumed to be invariant only for a time interval $(0,b)$, or even for a sequence of times $t_n$ converging to $0$.  In Sections \ref{subsec:short-times} and \ref{subsec:short-times-f-sign-changing}, we shall consider the first case. For the second case, in a different setting, compare with \cite{CP}. This means that we do not (want to) rely on the information for large times used in Theorem \ref{thm:rigidity-definitive}, which forced $f$ to be constant on $\pa\Om$. As a consequence of this choice, we still obtain spherical symmetry at the cost of assuming that $u$ admits \textit{two} time-invariant surfaces. The following result will be proved in Section \ref{subsec:short-times}.

\begin{thm}[Two invariant surfaces]
\label{thm:two-surfaces}
Let {$\Om\subset \mathbb R^N$} be a bounded domain. Let $f \in C(\pa\Om)$ be non-negative.
Assume that either
    \begin{enumerate}[(a)]
        \item $\Om$ satisfies the exterior sphere condition and $\cS=\pa\Om$, or
        \item $\pa\Om \in C^{1,1}$ and $\overline{\cS}=\pa\Om$.
    \end{enumerate}
Let $u$ be the solution of \eqref{heat-cauchy-dirichlet}.
\par
Suppose $D_1, D_2$ are two domains whose boundaries $\Ga_1$ and $\Ga_2\subset\Om$ are disjoint and satisfy the interior cone condition.
Assume that, for each $i=1, 2$, \eqref{time-invariant} holds  for $\Ga_i$  and some function $c_i:(0,b)\to (0,\infty)$. Then, $D_1, D_2$, and $\Om$ are concentric balls, and $f$ is constant on $\pa\Om$.
\end{thm}
\par
We conclude this introduction by presenting a non-existence result for invariant surfaces compactly contained in $\Om$, when $f$ is allowed to change sign. In what follows, we set
$$
 G=\{x\in \ol{\Om}:\, d_{\cS^+}(x)=d_{\cS^-}(x)\},
$$
where the sets $\cS^\pm$ are defined in \eqref{def-S12}, and
$$
G_\reg=\{x\in G: x \mbox{ has a unique closest point on } \pa\Om\}, \quad G_\sing=G\setminus G_\reg.
$$

The following result will be proved in Section \ref{subsec:short-times-f-sign-changing}.

\begin{thm}[Non-existence of invariant surfaces]
\label{thm:non-existence}
Let {$\Om\subset\RR^N$} be a bounded domain with boundary $\pa\Om$ of class $C^{1,1}$. Let $f\in C(\pa\Om)$ and suppose that it changes sign, i.e. the sets $\cS^\pm$ are non-empty and, also, $\ol{\cS}=\pa\Om$. 
\par
Assume that 
$G_\reg$ is the finite union of $k$ disjoint $(N-1)$-dimensional surfaces of class $C^1$ and the $(N-1)$-dimensional Hausdorff measure of $G_\sing$ is zero.
\par
Let $u=u(x,t)$ be the solution of \eqref{heat-cauchy-dirichlet}. If $\Ga=\pa D$, with $\ol{D}\subset\Om$, and $\pa D$ is of class $C^1$ and satisfies the interior sphere condition, then condition \eqref{time-invariant} is never satisfied for $a=0$.
\end{thm}
\par
Nevertheless, it should be noticed that, in the sign-changing regime, the existence of invariant surfaces which extend up to the boundary $\pa\Om$ is certainly possible. It is also possible that time-invariant ``surfaces'' with self-intersections occur.  See Remark \ref{rem:IS-sign-changing}. 
\par
The problem of characterizing this kind of surfaces is quite delicate and deserves a separate treatment, which will be the theme of a forthcoming paper.

%
%
%

\section{General assumptions and notations}
\label{sec:assumptions}

In this section, we collect the main ingredients and preliminary results needed for the sequel of the paper.

\subsection{The matrix $A(x)$}
\label{subsec:A}
The following assumptions are made on the symmetric matrix of functions $A(x)$:
\begin{enumerate}[(a)]
\item
$A$ is differentiable and locally Lipschitz continuous in $\RR^N$ and, for any compact set $K$,  there is a constant $L_K>0$  such that
$$
|A(x)|, |\na A(x)|\le L_K\ \ \mbox{ for any } \ x\in K.
$$
\item
there is a positive constant $\la$ such that
$$
\bigl\lan A(x) \xi, \xi\bigr\ran  \ge \la\,|\xi|^2 \ \mbox{ for any } \ x, \xi\in\RR^N.
$$
\end{enumerate}

\subsection{The domain $\Om$}
\label{subsec:Omega}
We consider a bounded domain $\Om\subset\RR^N$, $N\ge 2$. For $r>0$, we set:  
$$
\Om_r=\{ y\in\RR^N : \dist(y,\Om)<r\} \ \mbox{ and } \ \Om^r=\{ y\in\RR^N : \dist(y,\Om)\ge r\}.
$$
Also, $B_R$ will denote the ball of radius $R$ centered at the origin,  and, for two sets $E$ and $F$, we define:
$$
E+F=\{x+y: x\in E, y\in F\},
$$
the so-called \textit{Minkowski sum} of $E$ and $F$. Thus, we have that $\Om_r=\Om+B_r$.\par
We know that, the boundary $\pa\Om$ is of class $C^{1,1}$ if and only if  $\pa\Om$ satisfies both the interior and exterior uniform sphere condition. Occasionally, the relevant radii in these conditions will be denoted by $r_i$ and $r_e$.   
If $r<\min(r_i, r_e)$, then also $\Om^r$ satisfies both the interior and exterior uniform sphere condition (with smaller radius).
\par
When $\pa\Om$ is of class $C^2$, we shall denote by $\ka_1,\cdots, \ka_{N-1}$ the principal curvatures of $\pa\Om$ with respect to the interior normal $\nu$ at points of $\pa\Om$.

\subsection{The distance $d_\cS$}
\label{subsec:distance}
We define an intrinsic distance of points in $\Om$ to subsets of $\pa\Om$, which takes into account the metric induced by $A(x)$ and the shape of $\Om$. 
\par
In fact, we first define a point dependent inner product by
$$
\bigl\lan w_1, w_2\bigr\ran_A=\bigl\lan A^{-1}(x)\, w_1, w_2 \bigr\ran, \ w_1, w_2\in\RR^N,
$$ 
where $A^{-1}(x)$ is the inverse matrix of $A(x)$. The assumptions on $A(x)$ guarantee that the related norm $| w|_A$ is equivalent to the standard Euclidean norm.
\par
Next, we shall consider paths of class $H^1_\loc([0,\infty);\RR^N)$ and define, for $t\ge 0$ and $x\in\Om$, the number:
$$
T_x=\inf\bigl\{ \tau>t: \ga(\tau)\in\RR^N\setminus\Om, \ \ga(t)=x\bigr\},
$$
where the infimum of the empty set is $+\infty$.
This is the \textit{first exit time} from $\ol{\Om}$, 
after time $t$.
Then, if $\cS\subseteq\pa\Om$, we can define a distance as
\begin{equation}
\label{def-d-S}
d_\cS(x)=\inf\biggl\{
\int_0^{T_x} \bigl|\ga'(t)\bigr|_A\, dt:  \ga(0)=x, \ \ga(T_x)\in\cS \ \mbox{ if } \ T_x<\infty
\biggr\}, 
\end{equation}
for $x\in\ol{\Om}$.
 
\begin{rem}[\cite{EI,CDL}]
{\rm
We recall from \cite{EI} the following facts. See also \cite{CDL}, for other details.

\par
(i) Notice that it holds the alternative definition:
$$
d_\cS(x)=\inf\biggl\{
\frac12\int_0^{T_x} \Bigl(\bigl|\ga'(t)\bigr|_A^2+1\Bigr) dt:  \ga(0)=x, \ \ga(T_x)\in\cS \ \mbox{ if } \ T_x<\infty \biggr\}.
$$
\par
(ii)
The function $d_\cS$ is of class $C^{0,1}(\ol{\Om})$ and is a viscosity solution of
the eikonal equation, in the relevant metric, 
\begin{equation*}
\bigl\lan A(x) \na v, \na v\bigr\ran-1=0 \ \mbox{ in } \ \Om,
\end{equation*}
such that $d_\cS=0$ on $\ol{\cS}$.
\par
The proof of this formula can be obtained by merging the arguments used in \cite[Examples 1 and 2]{EI}. In particular, we can proceed as in \cite[Lemma 1.3]{EI}, with $\la=1/2$ and the obvious modification regarding the constraint $\ga(T_x)\in\cS$, which appears in \cite[Example 2]{EI}. 
}
\end{rem}

\begin{rem}
\label{rem:d-r}
{\rm
Set 
$$
\cS_r=\{ y\in\pa\Om_r : |x-y|<r \ \mbox{ for any } \ x\in\cS\}.
$$
It is clear that $\dist(x,\cS_r)<r$, for any $x\in\cS$, and  $\dist(y,\cS)<r$ for any $y\in\cS_r$.
Also, if $d_{\cS_r}$ is defined as in \eqref{def-d-S}, with $\cS$ replaced by $\cS_r$, we have that
$d_{\cS_r}\le C^*\,r$ on $\cS$, for some constant $C^*$.
}
\end{rem}

\subsection{The sets $\cS$, $\cS^\pm$, $\Om^\pm$, and $G$}
\label{subsec:def-S}
We take $f\in C(\pa\Om)$ and define
$$
\ul{f}=\min_{\pa\Om} f, \quad \ol{f}=\max_{\pa\Om} f.
$$
\par
As in \eqref{def-S12} we set:
\begin{equation*}
\cS^+=\{ x\in\pa\Om : f(x)>0\}, \quad \cS^-=\{ x\in\pa\Om : f(x)<0\},
\end{equation*}
and $\cS=\cS^+\cup\cS^-$. Notice that, when $f\ge 0$, then
\begin{equation}
\label{def-S}
\cS=\{ x\in\pa\Om : f(x)>0\},
\end{equation}
so that its closure $\ol{\cS}$ is the usual support of $f$.
Moreover, we define the sets
 $$
 \Om^\pm =\{x \in \Om:\, \pm\left[d_{\cS^+}(x)-d_{\cS^-}(x)\right] < 0\}.
$$
Finally, we recall from the introduction the definition:
$$
G=\{x\in \ol{\Om}:\, d_{\cS^+}(x)=d_{\cS^-}(x)\}.
$$


The next result provides one case of sufficient conditions for which the interior of $G$ is empty.

\begin{lem}
\label{lem: sufficient condition}
    Let $\Om$ be a convex domain in $\RR^N$ and let $\cS^+$ and $\cS^-$ be subsets of $\pa\Om$ such that $\ol{\cS^+}\cap\ol{\cS^-}=\varnothing$.
Then $G$ contains no open sets. 
In particular, we have that $G=\overline{\Om^+}\cap \overline{\Om^-}$ and $\Om=\overline{\Om^+}\cup\overline{\Om^-}$.
\end{lem}

\begin{proof}
(i) Take $x\in G$. There exist $y^\pm\in \cS^\pm$ such that $
|x-y^+|=d_{\cS^+} (x)=d_{\cS^-} (x)=|x-y^-|$. 
\par
Set $r=|x-y^\pm|$. We shall show
that, for any $\de\in (0,1)$,  the point $x_\de= x + \de (y^+- x)$ does not belong to $G$. Indeed, we easily
compute that $d_{\cS^+}(x_\de)\le|x_\de- y^+| = r\,(1-\de)$. On the other hand, if
$y_\de^-$ denotes the point in $\cS^-$ such that $d_{\cS^-}(x_\de)= |x_\de-y^-_\de|$, by the triangular inequality we infer that 
$$
|x_\de-y_\de^-|\ge |x-y_\de^-|-|x-x_\de|\ge r-r\,\de=|x_\de-y^+|.
$$
\par
Thus, a straightforward analysis of equality cases in the triangular inequality gives 
that the first inequality is strict unless $y^-= y^+$. Hence, $\ol{\cS^+}\cap\ol{\cS^-}\ne\varnothing$ --- a contradiction.
\par
(ii) An alternative proof. Suppose the interior $G'$ is non-empty. Being $G'$ an open set, we infer that 
$$
\na d_{\cS^+}=\na d_{\cS^-},
$$ 
almost everywhere in $G'$. Let $y\in G'$ a point at which $d_{\cS^+}$ and $d_{\cS^-}$ are differentiable and $\na d_{\cS^+}(y)=\na d_{\cS^-}(y)$. If $z^\pm\in\cS^\pm$ are points such that
$$
d_{\cS^\pm}(y)=|y-z^\pm|,
$$
then $|y-z^+|=|y-z^-|$, being as $d_{\cS^+}(y)=d_{\cS^-}(y)$, and
$$
\frac{y-z^+}{|y-z^-|}=\frac{y-z^+}{|y-z^+|}=\na d_{\cS^+}(y)=\na d_{\cS^-}(y)=\frac{y-z^-}{|y-z^-|}.
$$
This gives that $z^+=z^-$, i.e. $\ol{\cS^+}\cap\ol{\cS^-}\ne\varnothing$ --- a contradiction.
\end{proof}

\begin{ex}[The interior of $G$ may be non-empty]
\label{ex:fat-G}
{\rm
Here, we choose $A(x)=I$, so that the relevant distance is induced by the Euclidean distance of points. The two pictures in Figure 1 below give examples in which the set $G$ contains an open subset of $\Om$.

\begin{figure}[hbt]
\begin{center}
\begin{tikzpicture}[scale=.21]

\draw[very thick, black] (-21,0) arc (0.8:360:8); 


\fill [yellow, draw=yellow] (-21,0) -- (-34.6,-5.7) -- (-29,0) -- cycle;

\fill [yellow, draw=brown] plot [variable=\t, domain=225:360, samples=50]
({-29+8*cos(\t)}, {8*sin(\t)});


\draw[ultra thick, red] (-21,0) arc (0.8:90:8); 
\draw[ultra thick, green] (-21,0) arc (0.8:73:8); 
\draw[ultra thick, red] (-21,0) arc (0.8:59:8); 
\draw[ultra thick, green] (-21,0) arc (0.8:45:8); 
\draw[ultra thick, red] (-21,0) arc (0.8:34:8); 
\draw[ultra thick, green] (-21,0) arc (0.8:25:8); 
\draw[ultra thick, red] (-21,0) arc (0.8:18:8); 
\draw[ultra thick, green] (-21,0) arc (0.8:11:8); 
\draw[ultra thick, red] (-21,0) arc (0.8:5:8);

\draw[thick, brown] (-29,0) -- (-21,0); 
\draw[thick, brown] (-34.6,-5.7) -- (-29,0);

\node [olive] at (-26,6.3) {$\mathbf{\cS^+}$};  
\node [red] at (-27.2,9.3) {$\mathbf{\cS^-}$};  

\node [black] at (-27.5,-3.5) {$\mathbf{G'}$};  
\node [black] at (-20,0) {$A$};  

\fill [yellow, draw=brown] [thick] plot [smooth] coordinates
{(-10,0) (0,0) (-1,-0.5) (-2,-5) (-4,-7) (-7,-4.5) (-9,-2)  (-10,0)};

\draw [very thick, green] plot [variable=\t, domain=323:342, samples=50]
({2+10*cos(\t)}, {-0.5+10*sin(\t)});

\draw [very thick, red] plot [variable=\t, domain=342:365, samples=50]
({2+10*cos(\t)}, {-0.5+10*sin(\t)});

\draw [thick, yellow] plot [smooth] coordinates {(-10, 0) (-0.2, 0)};

\draw [thick, black] plot [smooth] coordinates {(-10, 0) (-10.3, 2.5) (-9, 4) (-5, 6) (0,7) (5,7) (8,6)  (11, 3)  (12, 0.4)};



\draw [thick, black] plot [smooth] coordinates {(0,0) (2,-0.5) (3.8,-4) (7.5,-8) (10,-6.6)};


\node [black] at (-5,-3) {$\mathbf{G'}$};  
\node [olive] at (12.5,-5.2) {$\mathbf{\cS^+}$};  
\node [red] at (13.5,-1.8) {$\mathbf{\cS^-}$};  

\end{tikzpicture}
\end{center}
\caption{The set $G'$ is contained in $G$. {\it Left:} $\Om$ is convex and the sets $\cS^+$ and $\cS^-$ accumulate at the point $A$. {\it Right:} $\Om$ is not convex. The distance depends on the shape of $\Om$.}
\end{figure}
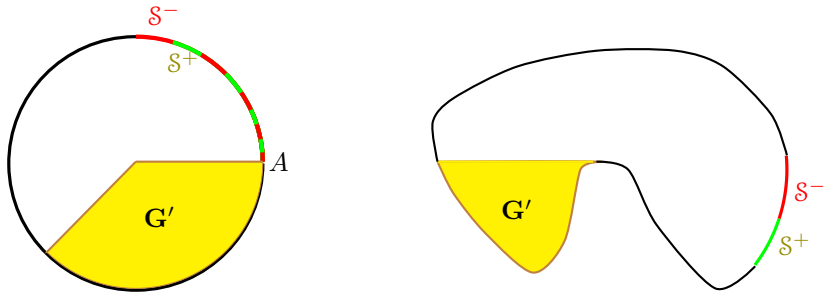

}
\end{ex}

\section{The Varadhan formulas with non-constant boundary values}
\label{sec:non-constant}

In this section, we shall prove our main result, i.e. Theorem \ref{thm:varadhan-elliptic-parabolic}. Of course, throughout this section, we will assume that the Cauchy data $u_0$ is identically equal to zero.
\par
Before entering the details of the proof in its general setting, we show that Theorem \ref{thm:varadhan-elliptic-parabolic} easily holds true in case of bounded and non-vanishing Dirichlet boundary values. This proposition is obtained by extending some arguments used in \cite[Theorem 3.7]{MS4}.

\begin{prop}[Easy extension of Varadhan's formulas]
\label{prop:varadhan-easy-extension}
Let $\Om$ be an open domain in $\RR^N$  with boundary $\pa\Om$ made of regular points for the Dirichlet problem.
Let  $f\in C(\pa\Om)$ be such that $a^-\le f\le a^+$ on $\pa\Om$, for some positive constants $a^-, a^+$.  
\par
Let $u=u(x,t)$ and $U_\ve$ be the solutions of the problems \eqref{cauchy-dirichlet}  and \eqref{resolvent-dirichlet}. Then, formulas \eqref{varadhan} and \eqref{varadhan-elliptic} hold true, uniformly on $\ol{\Om}$.
\end{prop}

\begin{proof}
We provide the details of the proof of \eqref{varadhan}. The proof of \eqref{varadhan-elliptic} runs similarly. 
Note, in passing, that a point is regular for the elliptic operator in \eqref{resolvent-dirichlet} if and only if  is regular for the Laplace operator (see \cite{LSW}).  
\par
Let $u^-$ and $u^+$ be the solutions of \eqref{cauchy-dirichlet} with $f$ replaced by $a^-$ and $a^+$. Hence, by uniqueness, $u^-/a^-$ and $u^+/a^+$ coincide with the solution $u^*$ of \eqref{cauchy-dirichlet} with $f\equiv 1$.
Also, by comparison we deduce that
$$
u^-\le u\le u^+ \ \mbox{ in } \ \Om\times (0,\infty),
$$
and hence 
$$
a^- u^*\le u\le a^+ u^* \ \mbox{ in } \ \Om\times (0,\infty).
$$
Thus, by \eqref{varadhan}, we conclude that
\begin{multline*}
-d_{\pa\Om}(x)^2=\lim_{t\to 0^+} 4t \log(f^- u^*(x,t))\le \liminf_{t\to 0^+} 4t \log u(x,t)\le \\ \limsup_{t\to 0^+} 4t \log u(x,t)\le\lim_{t\to 0^+} 4t \log(f^+ u^*(x,t))=-d_{\pa\Om}(x)^2,
\end{multline*}
for $x\in\ol{\Om}$. The uniform convergence follows from that in \cite{Va}.
\end{proof}

\subsection{The elliptic problem}

We start by recalling that the two problems \eqref{cauchy-dirichlet}  and \eqref{resolvent-dirichlet} are related by the formula \eqref{laplace}. In fact, a simple integration by parts shows that the latter problem follows from the former. 
Viceversa, the uniqueness theorem for the Laplace transformation gives that $u$ is uniquely determined by $U^\ve$.

We now proceed to consider problem \eqref{resolvent-dirichlet} in the case in which $f$ is non-negative and is allowed to vanish on $\pa\Om$.
We follow the arguments contained in \cite{EI}. In fact, we rescale the values of the solution $U_\ve$ of \eqref{resolvent-dirichlet} by a Hopf-Cole transformation, i.e. we
set  $V_\ve=-\ve \log U_\ve$. It is easy to see that $V_\ve$ satisfies the equation
\begin{equation}
\label{rescaled-equation}
-\ve\,\tr\bigl[ A(x)\,\na^2 V_\ve\bigr]+\bigl\lan A(x)\,\na V_\ve, \na V_\ve\bigr\ran-1=0 \ \mbox{ in } \ \Om,
\end{equation}
and is such that
\begin{equation}
\label{dirichlet-V-eps}
V_\ve=-\ve\,\log f \ \mbox{ on } \ \cS, \quad V_\ve=+\infty  \ \mbox{ on } \ \pa\Om\setminus\cS.
\end{equation}

The next step in our agenda is to derive an Arzel\`a compactness estimate for the family of solutions $\{V_\ve\}_{\ve>0}$ with special Dirichlet boundary values.

\begin{lem}[Compactness of $\{ V_\ve\}_{\ve>0}$]
\label{lem:compactness-estimates}
Let $\Om$ be a bounded domain in $\RR^N$, with boundary of class $C^{1,1}$, and $\cS$ be a non-empty relatively open  subset of $\pa\Om$.
\par
For any fixed $\ve>0$, let  $U_\ve$ be the bounded solution of  \eqref{resolvent-dirichlet}, with 
\begin{equation}
\label{step-f}
f=f_0\,\one_\cS \ \mbox{ on } \ \pa\Om,
\end{equation}
for some constant $f_0>0$.
Set $V_\ve=-\ve \log U_\ve$. Then, for any open subset $\Om'$ of $\Om$ such that $\ol{\Om'}\subset\Om\cup\cS$, there exists a positive constant $C_{\Om'}$ such that
$$
\max_{\ol{\Om'}} \bigl( |V_\ve|+|\na V_\ve|\bigl)\le C_{\Om'},
$$
for every $\ve>0$.
\par
In particular, up to subsequences, $V_\ve$ converges uniformly on compact subsets of $\Om\cup\cS$ to a continuous function $V$. 
Moreover, $V$ can be uniquely extended to a function of class $C^{0,1}(\ol{\Om})$, which is a viscosity solution of 
\begin{equation}
\label{eikonal-A}
\bigl\lan A(x) \na V, \na V\bigr\ran=1 \ \mbox{ in } \ \Om,
\end{equation}
such that $V=0$ on $\cS$.
\end{lem}

\begin{proof}
The bounds are substantially obtained by combining the arguments contained in \cite[Examples 1 and 2]{EI}. 
For notational convenience, in what follows, we drop the subscript $\ve$. Then, we consider a subdomain $\Om'$ of $\Om$, as specified in the statement of this lemma and
set $\cS'=\pa\Om'\cap\cS$. 
\par
(i) A barrier argument shows that there is a constant $C'$ such that
$$
\sup_{\cS'}|\na V|\le C'.
$$
\par
(ii) Here, we use the Bernstein method in order to obtain a local interior bound for $|\na V|$. For the reader's convenience, we show only the details in the case in which $A(x)$ is the identity matrix, in order to avoid unnecessary technicalities in the exposition. The interested reader is referred to \cite{EI}, for the details of the general case, when $A(x)$ is assumed of class $C^1$. 
\par
In fact, consider a smooth cutoff function $\eta$ with support in $\Om'\cup\cS'$ and such that $0\le\eta\le 1$, and set
$$
W=\eta^4 |\na V|^2.
$$
It is clear that the maximum of $W$ is attained either in $\Om'\cap\Om$ or on $\cS'$ (where it is bounded by $C'$). Thus, we suppose that $W$ attains its maximum at a point $z\in\Om'\cap\Om$. Then, at $z$ we have that
$$
0=\na W=2\,\eta^4 (\na^2 V) \na V+4\,\eta^3 |\na V|^2 \na\eta
$$
and
\begin{multline*}
0\ge\De W= 
2\,\eta^4 \bigl[|\na^2 V|^2+ 
\lan\na V,\na(\De V)\ran\bigr]+ \\
16\,\eta^3 \lan (\na^2 V) \na V,\na\eta\ran+12\,\eta^2 |\na V|^2 |\na\eta|^2+4\,\eta^3 |\na V|^2 \De\eta= \\
2\,\eta^4 \bigl[|\na^2 V|^2+\lan\na V,\na(\De V)\ran\bigr]-20\,\eta^2 |\na V|^2 |\na\eta|^2
+4\,\eta^3 |\na V|^2 \De\eta.
\end{multline*}
In the last equality, we used that $\na W(z)=0$.
\par
On the other hand, thanks to \eqref{rescaled-equation}, we have that
$$
0=\na \bigl[ -\ve\,\De V+|\na V|^2-1]=-\ve\,\na(\De V)+2\, (\na^2 V) \na V
$$
in $\Om$, and hence, by differentiating, we get:
$$
\ve\, \lan\na V,\na(\De V)\ran=2\, \lan (\na^2 V) \na V,\na V\ran.
$$
Combining this identity with the above conditions on $\na W$ and $\De W$ at $z$ then gives that
\begin{multline*}
2\,\ve^2 \eta^2\bigl[5\,|\na\eta|^2-\eta\,\De\eta\bigr] |\na V|^2 \ge 
 \ve^2 \eta^4 \bigl[|\na^2 V|^2+\lan\na V,\na(\De V)\ran\bigr]= \\
\ve^2 \eta^4 |\na^2 V|^2+2\ve\,\eta^4 \lan (\na^2 V) \na V,\na V\ran=
\ve^2 \eta^4 |\na^2 V|^2+4\ve\,\eta^3 |\na V|^2 \lan \na V, \na\eta\ran\ge \\
\frac1{N}\, \eta^4 \bigl(|\na V|^2-1\bigr)^2+4\ve\,\eta^3 |\na V|^2 \lan \na V, \na\eta\ran.
\end{multline*}
In the last inequality, we used the equation satisfied by $V$ and the pointwise inequality:
$$
N\,|\na^2 V|^2\ge \bigl\lan\na^2 V, I\bigr\ran^2=(\De V)^2.
$$
\par
All in all, for $0<\ve<1/(2 N)$, we obtain that
$$
\eta^4 \bigl(|\na V|^2-1\bigr)^2\le \nr\na\eta\nr_\infty\,|\na V|^3+\bigl( 5\,\nr\na\eta\nr_\infty^2+\nr\De\eta\nr_\infty\bigr) |\na V|^2.
$$
Hence we get a uniform bound for $\eta\, |\na V|$ at $z$.
Thus, thanks to item (i), we can conclude that $|\na V_\ve|^2$ is locally uniformly bounded on any $\ol{\Om'}\subset\Om\cup\cS$. 
Also, since $V_\ve=-\ve \log f_0$ on $\cS$, by the gradient estimate, we easily get that $|V_\ve|$ is locally uniformly bounded on any $\ol{\Om'}\subset\Om\cup\cS$.
\par 
(ii) The uniform convergence up to subsequences follows from the Arzel\`a theorem. Clearly, $V=0$ on $\cS$, being as $V_\ve=-\ve \log f_0$ on $\cS$, for every $\ve>0$.
Moreover, by the standard stability result for viscosity solutions (i.e. ``uniform limits of solutions converge to solutions of uniform limits of equations'', see e.g. \cite{CIL}), $V$ is a viscosity solution of the eikonal equation \eqref{eikonal-A}.
\par 
This fact gives that $|\na V|$ is bounded on the whole $\Om$ (see \cite[\S \ I.4]{CL}), and hence $V\in C^{0,1}(\ol{\Om})$, in the sense that $V$ has a unique Lipschitz extension to $\ol{\Om}$.
\end{proof}

In the next lemma, by a standard convolution technique, we show that $d_\cS$ can be regularised in a way that the new function is a subsolution of \eqref{rescaled-equation}, up to an error vanishing with $\ve$. In what follows, by $\{ \eta_\ve\}_{\ve>0}$ we denote a mollifying kernel defined as

\begin{enumerate}[(a)]
\item
$\eta\in C^\infty_0(\RR^N)$, with $\eta\ge 0$ in $\RR^N$,  
$\supp(\eta)\subset B_1(0)$, and $\int_{\RR^N}\eta\,dy=1$;
\item
$\eta_\ve(y)=\ve^{-\frac{N}{2}} \eta(y/\sqrt{\ve})$ for $0<\ve<r$.
\end{enumerate}
It is not difficult to see that
\begin{equation}
\label{asymptotics-for-eta-ep}
\int_{\RR^N} |y|\,\eta_\ve(y)\,dy=O(\sqrt{\ve})
\ \mbox{ and } \ \int_{\RR^N} |\na\eta_\ve(y)|\,dy=O(1/\sqrt{\ve}) \ \mbox{ as } \ \ve\to 0.
\end{equation}

\begin{lem}[Regularising $d_\cS$]
Let $\cS$ be a subset of $\pa\Om$ and, for $0<r<R$, let $\cS_r$ be defined as in Remark \ref{rem:d-r}. For $a>0$, define in $\Om_r$ the functions
$$
d_{a,r}=\frac{1-(1-a\,d_{\cS_r}/2)^2}{a}=d_{\cS_r}-\frac{a}{4}\,d_{\cS_r}^2
$$
and
$$
d_a^\ve=\eta_\ve\star d_{a,r}=\eta_\ve\star (d_{a,r}\one_{\Om_r}) \ \mbox{ in } \ \RR^N.
$$
Then, $d_a^\ve$ is a solution of 
\begin{equation}
\label{eq-eps-d-a}
a\,d_a^\ve-\ve\,\tr\bigl[ A\,\na^2 d_a^\ve\bigr]+\bigl\lan A\,\na d_a^\ve, \na d_a^\ve\bigr\ran-1\le O(\sqrt{\ve}) \ \mbox{ in } \ \Om \ \mbox{ as } \ \ve\to 0^+,
\end{equation}
where $O(\sqrt{\ve})$ is uniform.
\par
Moreover,
$$
d_a^\ve\le r  \ \mbox{ on } \ \cS \ \mbox{ and } \ d_a^\ve\le\max_{\pa\Om} d_{\cS} \ \mbox{ on } \ \pa\Om\setminus\cS.
$$
\end{lem}

\begin{proof}
(i)
Since $0<d_a^\ve\le d_{\cS_r}$, we easily infer that $d_a^\ve\le \nr\eta_\ve\nr_1 \nr d_{a,r}\nr_\infty\le \nr d_{\cS_r}\nr_\infty$, so that the bounds on $\cS$ easily ensue.
Moreover, 
it is {straightforward} to check that $d_{a,r}$ is a viscosity solution of
$$
a\,d_{a,r}+\bigl\lan A\,\na d_{a,r}, \na d_{a,r}\bigr\ran-1=0 \ \mbox{ in } \ \Om_r.
$$
Since $d_{a,r}$ is Lipschitz continuous, we can assume that the last equation is satisfied almost everywhere in $\Om_r$.
Thus, the convolution with $\eta_\ve$ gives that
$$
a\,d^\ve_a+\eta_\ve\star\bigl\lan A\,\na d_{a,r}, \na d_{a,r}\bigr\ran-1=0.
$$
Hence, at $x$  we deduce that
$$
a\,d_a^\ve+\bigl\lan A\,\na d_a^\ve, d_a^\ve\bigr\ran-1=
\bigl\lan A\,\na d_a^\ve, \na d_a^\ve\bigr\ran-\eta_\ve\star \bigl\lan A\,\na d_{a,r}, \na d_{a,r}\bigr\ran.
$$
\par
(ii)
We now show that the right-hand side in the last formula is bounded from above by a $O(\sqrt{\ve})$ as $\ve\to 0^+$. \par
We start by observing that, since $d_{a,r}$ is Lipschitz continuous in $\Om_r$, then we have that
$$
\na d_a^\ve =\eta_\ve\star \na d_{a,r} \ \mbox{ and } \ \na_{ij} d_a^\ve =\na_i\eta_\ve\star \na_j d_{a,r} \ \mbox{ for} \  i, j=1,\dots, N.
$$
\par
Next, we easily infer that
$$
\bigl\lan A\,\na d^\ve_a, \na d^\ve_a\bigr\ran=\bigl| \sqrt{A}\,\na d^\ve_a\bigr|^2=
\bigl| \sqrt{A}\,\bigl(\eta_\ve\star \na d_{a,r}\bigr)\bigr|^2,
$$
where $\sqrt{A}$ is the square root of the matrix $A$. Thus, we infer that
$$
\bigl\lan A(x)\,\na d^\ve_a(x), \na d^\ve_a(x)\bigr\ran=
\Bigl| \bigl[\eta_\ve\star \sqrt{A(x)}\,\na d_{a,r}\bigr](x)\Bigr|^2.
$$
Note that, in the last convolution, $\sqrt{A(x)}$ is treated as a constant matrix.
\par
We thus observe that Jensen's inequality gives that
$$
\Bigl| \bigl[\eta_\ve\star \sqrt{A(x)}\,\na d_{a,r}\bigr](x)\Bigr|^2\le
\Bigl[\eta_\ve\star \bigl| \sqrt{A(x)}\,\na d_{a,r}\bigr|^2\Bigr](x),
$$
and hence we deduce that
$$
\bigl\lan A\,\na d^\ve_a, \na d^\ve_a\bigr\ran\le\eta_\ve\star \bigl\lan A(x)\,\na d_{a,r}, \na d_{a,r}\bigr\ran.
$$
As before, here $A(x)$ is treated as a constant matrix in the convolution.
\par
As a consequence, we can write that
\begin{multline*}
\bigl\lan A(x)\,\na d_a^\ve(x), \na d_a^\ve(x)\bigr\ran-\eta_\ve\star \bigl\lan A\,\na d_{a,r}, \na d_{a,r}\bigr\ran(x)\le \\
\int_{\RR^N} \eta_\ve(y)\Bigl\lan \bigl[ A(x)-A(x-y)\bigr]\na d_{a,r}(x-y), \na d_{a,r}(x-y)\Bigr\ran\, dy. 
\end{multline*}
Therefore, since $A$ and $d_{a,r}$ are Lipschitz continuous, we can estimate the modulus of the integral at the right-hand side by
$$
L\,\nr\na d_{a,r}\nr_\infty^2 \int_{\RR^N} |y|\,\eta_\ve(y)\,dy=O(\sqrt{\ve}),
$$
thanks to \eqref{asymptotics-for-eta-ep}. This gives the desired conclusion of this item, and hence we deduce that
$$
a\, d^\ve_a +\bigl\lan A\,\na d^\ve_a, \na d^\ve_a\bigr\ran-1\le
O(\sqrt{\ve}) \ \mbox{ as } \ \ve\to 0.
$$
\par
(iii)
Therefore, in view of the last formula, in order to prove \eqref{eq-eps-d-a}, we are only left to show that 
$$
\bigl|\ve\,\tr\bigl[ A\,\na^2 d_a^\ve\bigr]\bigr|=O(\sqrt{\ve}).
$$
This is readily obtained by observing that
$$
\bigl|\tr\bigl[ A\,\na^2 d_a^\ve\bigr]\bigr|\le L\,\nr\na\eta_\ve\nr_1 \nr\na d_{a,r}\nr_\infty=O(1/\sqrt{\ve}),
$$
thanks to \eqref{asymptotics-for-eta-ep}. The proof is complete.
\end{proof}


The next lemma is crucial in order to prove the elliptic part of Theorem \ref{thm:varadhan-elliptic-parabolic}.

\begin{lem}
\label{lem:characteristic-function}
Let $\Om$ be a bounded domain in $\RR^N$ with boundary of class $C^{1,1}$ and $\cS$ be a relatively open non-empty subset of $\pa\Om$. 
\par
Let  $U_\ve$ be the solution of the Dirichlet problem \eqref{resolvent-dirichlet}, with $f$ such that
$$
f=f_0\,\one_\cS \ \mbox{ on } \ \pa\Om,
$$
for some constant $f_0>0$ (as in \eqref{step-f}). Then, it holds that
$$
\lim_{\ve\to 0^+} \ve \log U_\ve=-d_\cS,
$$
uniformly in $\Om\cup\cS$.
\end{lem}

\begin{proof}
(i) We have that $U_\ve$ satisfies the assumptions of Lemma \ref{lem:compactness-estimates}. Thus, we can consider one of the uniform limits $V$ mentioned in that lemma. 
\par
Then, fix an $x\in\Om$. Also, let $\ga\in H^1_\loc([0,\infty), \RR^N)$ be such that $\ga(0)=x$ and $\ga(T_x)=y$. Then, the function $V(\ga(t))$ is absolutely continuous and hence
$$
\bigl|V(x)\bigr|=\bigl|V(\ga(T_x))-V(\ga(0))\bigr|=\left|\int_0^{T_x} \frac{d}{dt}V(\ga(t))\,dt\right|\le\int_0^{T_x} \bigl|\bigl\lan \na V(\ga(t)), \ga'(t)\bigr\ran\bigr| dt.
$$
The second equality follows from Rademacher's theorem, being as $V$ Lipschitz continuous.
\par
Now, let $Q(x)$ be the square root of the matrix $A(x)$. We have that 
\begin{multline*}
\bigl|\bigl\lan \na V(x), \ga'(t)\bigr\ran\bigr|=\bigl|\bigl\lan Q^{-1}(x) Q(x)\na V(x), \ga'(t)\bigr\ran\bigr|=\\
\bigl|\bigl\lan Q(x)\na V(x), Q^{-1}(x)\,\ga'(t)\bigr\ran\bigr|\le
\bigl| Q(x)\na V(x)\bigr| \bigl| Q^{-1}(x)\,\ga'(t)\bigr| =\\
\sqrt{\bigl\lan A(x) \na V(x), \na V(x)\bigr\ran} \sqrt{\bigl\lan A(x)^{-1} \ga'(t), \ga'(t)\bigr\ran}\le \\ 
\sqrt{\bigl\lan A(x)^{-1} \ga'(t), \ga'(t)\bigr\ran},
\end{multline*}
at $x=\ga(0)$. In the last inequality, we used \eqref{eikonal-A}.
This means that 
$$
|V(x)|=\left|\int_0^{T_x} \bigl\lan \na V(\ga(t)), \ga'(t)\bigr\ran dt\right| \le \int_0^{T_x}|\ga'(t)|_A \,dt
$$
and, by taking the infimum on all the admissible paths, we infer that
$$
|V(x)|\le d_\cS(x).
$$
In particular, $V(x)\le d_\cS(x)$. (See also \cite{CDL}.)
\par
(ii) Let $V_r^\ve$ denote the solution of the problem:
\begin{eqnarray*}
\label{eq-d-a-eps}
&&a\,V_r^\ve-\ve\,\tr\bigl[ A\,\na^2 V_r^\ve\bigr]+\bigl\lan A\,\na V_r^\ve, \na V_r^\ve\bigr\ran-1=0 \ \mbox{ in } \ \Om, \\
&&V_r^\ve=C^*\,r \ \mbox{ on } \ \cS, \quad V_r^\ve=\max_{\pa\Om\setminus\cS}d_\cS \ \mbox{ on } \ \pa\Om\setminus\cS.
\end{eqnarray*}
\par
We compare $V_r^\ve$ to the solution $V^\ve$ of problem 
\eqref{rescaled-equation}-\eqref{dirichlet-V-eps}, with $f$ as in \eqref{step-f}. Without loss of generality, we can assume that $f_0=1$, so that $V^\ve=0$ on $\cS$. Notice that $V^\ve+C^* r$ is also a solution of \eqref{rescaled-equation} and we have that
$$
V^\ve+C^*r\ge V_r^\ve \ \mbox{ on } \ \pa\Om.
$$
Moreover, since $V_r^\ve>0$ in $\Om$, we have that $V_r^\ve$ is a subsolution of \eqref{rescaled-equation}. Thus, by comparison, we deduce that
$$
V^\ve+C^* r\ge V_r^\ve \ \mbox{ on } \ \ol{\Om}.
$$
\par
Now, we compare $V_r^\ve$ to $d_a^\ve$. By construction, we have that
$$
d_a^\ve\le V_r^\ve<V_r^\ve+\frac{O(\sqrt{\ve})}{a} \ \mbox{ on } \ \pa\Om.
$$
Moreover, we notice that $V_r^\ve+O(\sqrt{\ve})/a$ is a supersolution of the same equation, \eqref{eq-eps-d-a}, for $d_a^\ve$. Thus, by comparison, we infer that
$$
d_a^\ve\le V_r^\ve+\frac{O(\sqrt{\ve})}{a} \ \mbox{ on } \ \ol{\Om}.
$$
All in all, we deduce that
$$
d_a^\ve\le V^\ve+C^* r+\frac{O(\sqrt{\ve})}{a} \ \mbox{ on } \ \ol{\Om}.
$$
\par
Letting $\ve\to 0$ then gives that $d_{a,r}\le V+C^* r$ on $\ol{\Om}$, and hence we get that
$d_\cS\le V$ on $\ol{\Om}$, in the limit as $a$ and $r$ tend to zero (see \cite{EI}). Therefore, $V=d_\cS$, thanks to item (i).
\par
(iii) What we have just proved is that any subsequence of $V_\ve$ converges (uniformly) to $d_\cS$ on $\Om\cup\cS$. Hence the whole sequence $V_\ve$ converges to $d_\cS$ on $\Om\cup\cS$. 
\end{proof}

The following result settles the elliptic part of Theorem \ref{thm:varadhan-elliptic-parabolic}.
\begin{thm}
\label{thm:varadhan-any-positive}
Let $\Om$ be a bounded domain in $\RR^N$, with boundary of class $C^{1,1}$. Let $f\in C(\pa\Om)$ be non-negative on $\pa\Om$ and let $\cS$ be given by \eqref{def-S}.
\par
Let  $U_\ve$ be the solution of the Dirichlet problem \eqref{resolvent-dirichlet}.
Then, it holds that
$$
\lim_{\ve\to 0^+} \left[-\ve\,\log U_\ve(x)\right]=d_\cS(x) \ \mbox{ for any } \ x\in\Om\cup\cS,
$$
and the convergence is uniform on the compact subsets of $\Om\cup\cS$. 
\end{thm}

\begin{proof}
The formula is clearly true when $x\in\cS$.
\par
Next, let $\ol{U}_\ve$ and $U_{\ve, n}$ be the solutions of \eqref{resolvent-dirichlet} with $f=\ol{f}\,\one_\cS$ and 
$$
f=\frac1{n}\,\one_{\cS_n},
$$
where $\cS_n=\{ x\in\pa\Om: f(x)>1/n\}$, for $n\in\NN$. Since $\cS_n\subset\cS$, we have that $d_{\cS_n}\ge d_\cS$. Also, it holds that
$$
U_{\ve, n}\le U_\ve\le \ol{U}_\ve \ \mbox{ on } \ \ol{\Om},
$$
by comparison, being as 
$$
\frac1{n}\,\one_{\cS_n}\le f\le \ol{f}\,\one_{\cS} \ \mbox{ on } \ \pa\Om.
$$
\par
Thus, we can apply Lemma \ref{lem:characteristic-function} and obtain that
\begin{multline*}
-d_{\cS_n}=\lim_{\ve\to 0^+} \ve\log U_{\ve, n}\le \liminf_{\ve\to 0^+} \ve\log U_\ve \le \\
\limsup_{\ve\to 0^+} \ve\log U_\ve\le 
\lim_{\ve\to 0^+} \ve\log\ol{U}_\ve=-d_\cS 
\end{multline*}
in $\Om$ and for any $n\in\NN$.
\par
Now, observe that $\cS$ is the union of the increasing sequence of sets $\cS_n$. Thus, pick any $x\in\Om$ and $\de>0$ . There is a smooth path $\ga$, with $\ga(0)=x$ and $\ga(T_x)\in\cS$, such that 
$$
d_\cS(x)+\de>\int_0^{T_x} \bigl|\ga'(t)\bigr|_A\, dt
$$ 
Also, there is an index $\nu\in\NN$ such that $\ga(T_x)\in\cS_\nu$. Hence, we deduce that
$$
d_{\cS_\nu}(x)\le \int_0^{T_x} \bigl|\ga'(t)\bigr|_A\, dt<d_\cS(x)+\de.
$$ 
All in all, we have that
$$
-d_\cS(x)-\de\le -d_{\cS_\nu}(x)\le \liminf_{\ve\to 0^+} \ve\log U_\ve(x) \le \\
\limsup_{\ve\to 0^+} \ve\log U_\ve(x)\le-d_\cS(x).
$$
Since $\de$ is arbitrary, we obtain the desired conclusion. By monotonicity, the Dini's monotone convergence theorem gives the desired uniform convergence.
\end{proof}

\subsection{The parabolic problem}
In this section, we shall settle the parabolic part of Theorem \ref{thm:varadhan-elliptic-parabolic}. In fact, it is sufficient to apply a lemma for the Laplace transform, \cite[Theorem 4.7]{Va}, which we recall in Lemma \ref{lem:laplace-transform}, for the reader's convenience.

\begin{thm}
\label{thm:non-constant-parabolic}
Let $\Om$ be a bounded domain in $\RR^N$ with boundary of class $C^{1,1}$. Let $f\in C(\pa\Om)$ be non-negative and let $\cS$ be given by \eqref{def-S}.
\par
Let $u=u(x,t)$ be the solution of the Cauchy-Dirichlet problem \eqref{cauchy-dirichlet}.
Then, it holds that
$$
\lim_{\ve\to 0^+} \left[-4t\,\log u(x,t)\right]= d_{\cS}(x)^2, 
$$
for $x\in\Om\cup\cS$ and uniformly on compact subsets of $\Om\cup\cS$.
\end{thm}

\begin{proof}
For $\be=x\in\Om\cup\cS$, we take $F_\be(t)=u(x,t)/\ol{f}$. Since $0\le f\le\ol{f}$ on $\pa\Om$, by the maximum principle, we have that $0\le F_\be(t)\le 1$. Also, $F_\be(0)=0$ and $F_\be(t)$ is non-decreasing on $[0,\infty)$, being as $f\ge 0$ on $\pa\Om$. 
\par
Next, by choosing $\tau=\ve^{-2}$, thanks to Theorem \ref{thm:varadhan-any-positive}, we infer that
$$
\lim_{\tau\to\infty}\left[-\frac1{\sqrt{2 \tau}} \log G_\be(\tau) \right]=
\lim_{\ve\to 0}\left[-\frac{\ve}{\sqrt{2}} \log \frac{U_\ve(x)}{\ol{f}} \right]=\frac{d_\cS(x)}{\sqrt{2}}
$$
for $x\in\Om\cup\cS$. 
\par
Thus, we can apply Lemma \ref{lem:laplace-transform} and deduce that
$$
\lim_{\ve\to 0^+} \left[-4t\,\log u(x,t)\right]=\lim_{t\to 0^+} \left[-4t\,\log F_\be(t)\right]=
d_\cS(x)^2, \ x\in\Om\cup\cS.
$$
The convergence is uniform on the compact subsets of $\Om\cup\cS$, being so in Theorem \ref{thm:varadhan-any-positive}.
\end{proof}

\begin{proof}[Proof of Theorem \ref{thm:varadhan-elliptic-parabolic}]
Merge Theorems \ref{thm:non-constant-parabolic} and \ref{thm:varadhan-any-positive}.
\end{proof}

\subsection{Asymptotics for the heat content of spheres}
\label{subsec:heat-content}
In this section, we shall prove Theorem \ref{thm:mean-value-asymptotics}, which extends \cite[Theorem 2.3]{MS1}. Thus, we set $A(x)=I$. 
\par
Then, we consider the solution $u=u(x,t)$ of the Cauchy-Dirichlet problem \eqref{heat-cauchy-dirichlet}, i.e.
\begin{equation*}
u_t=\De u \ \mbox{ in } \ \Om\times(0,\infty), \ u=0  \ \mbox{ on } \ \Om\times\{ 0\}, \
 u=f \ \mbox{ on } \ \pa\Om\times(0,\infty),
\end{equation*}
and the corresponding family of solutions $U_\ve$ of the Dirichlet problems \eqref{heat-resolvent-dirichlet}, i.e.
\begin{equation*}
\ve^2 \De U_\ve-U_\ve=0 \ \mbox{ in } \ \Om,  \quad
 U_\ve=f \ \mbox{ on } \ \pa\Om.
\end{equation*}

The proof of Theorem \ref{thm:mean-value-asymptotics} relies on a couple of lemmas, for which we recall the notations  for the sets $\Om_r$ and $\Om^r$, given in 
Section \ref{subsec:Omega}.

\begin{lem}
\label{lem:asymptotics-delta}
Let $B_r(p)\subset\ol{\Om}$ and fix $\de>0$. Let $U_\ve$ be the solution of \eqref{heat-resolvent-dirichlet}. 
Then, for every $\de'<\de$, it holds that
$$
\int_{\pa B_r(p)\cap\Om^\de} U_\ve(x)\,dS_x=o(e^{-\frac{\de'}{\ve}}) \ \mbox{ as } \ \ve\to 0^+.
$$
\end{lem}

\begin{proof}
Notice that, since $d_\cS\ge d_{\pa\Om}$ on $\ol{\Om}$, we have that $d_\cS\ge\de$ on $\Om^\de$. The uniform convergence in Theorem \ref{thm:varadhan-any-positive} then gives that there exists $\ol{\ve}>0$ such that
$$
-\ve \log U_\ve \ge  d_\cS -\frac{\de-\de'}{2}\ge \frac{\de+\de'}{2}
\ \mbox{ on } \ \Om^\de,
$$
for $0<\ve<\ol{\ve}$. This inequality gives in particular that
$$
0\le U_\ve\le e^{-\frac{\de+\de'}{2\ve}}
\ \mbox{ on } \ \Om^\de,
$$
for $0<\ve<\ve_\eta$. Consequently, we deduce that
$$
0\le e^{\frac{\de'}{\ve}}\int_{\pa B_r(p)\cap\Om^\de} U_\ve(x)\,dS_x\le |\pa B_r(p)|\, e^{-\frac{\de-\de'}{2\ve}},
$$
for $0<\ve<\ol{\ve}$. This inequality gives the desired claim.
\end{proof}

In what follows, $\fhi_\ve$ and $\psi_\ve$ denote the solutions of \eqref{heat-resolvent-dirichlet} with $f=1$ in a ball $B_r$ centered at the origin and its exterior $\RR^N\setminus\ol{B_r}$, respectively. These can be explicitly computed and are given by
$$
\fhi_\ve(x)=\frac{\int_0^\infty e^{\frac{|x|}{\ve}\cos\te}(\sin\te)^{N-2}d\te}{\int_0^\infty e^{\frac{r}{\ve}\cos\te}(\sin\te)^{N-2}d\te} \ \mbox{ for } \ 0\le |x|\le r,
$$
and
$$
\psi_\ve(x)=\frac{\int_0^\infty e^{-\frac{|x|}{\ve}\cosh\te}(\sinh\te)^{N-2}d\te}{\int_0^\infty e^{-\frac{r}{\ve}\cosh\te}(\sinh\te)^{N-2}d\te} \ \mbox{ for } \ |x|\ge r.
$$
It is clear that $\psi_\ve$ is the unique solution vanishing at infinity of the relevant problem.
\par
Thanks to them, we can construct the solution of \eqref{heat-resolvent-dirichlet}  in an annulus of radii $r$ and $\rho$, with $r<\rho$, such that $f=1$ for $|x|=\rho$ and $f=0$ for $|x|=r$. In fact, we have:
\begin{equation}
\label{Phi-eps}
\Upsilon_\ve(x)=\frac{\fhi_\ve(x)-\fhi_\ve(r) \psi_\ve(x)}{1-\fhi_\ve(r) \psi_\ve(\rho)} \ \mbox{ for } \ r\le |x|\le\rho.
\end{equation}



\begin{lem} 
\label{lem:U-eps-*}
Let $\Om$ satisfy the uniform exterior sphere condition with radius $r_e$. Take  $q\in\pa\Om$ and let $B_R(p)\subset\Om$ be a ball such that
$\pa B_R(p)\cap\pa\Om=\{q\}$. 
Then, we can choose $\de>0$ such that the solution $U_\ve^*$ of \eqref{heat-resolvent-dirichlet} in $\Om$ with $f=\one_{\pa\Om\setminus B_\de(q)}$ is such that
$$
\lim_{\ve\to 0^+}\ve^{-\frac{N-1}{2}}\int_{\pa B_R(p)} U^*_\ve\, dS_x=0.
$$
\end{lem}

\begin{proof}
By Lemma \ref{lem:asymptotics-delta}, for $\de>0$, we infer that
$$
\int_{\pa B_R(p)} U^*_\ve\, dS_x=\int_{\pa B_R(p)\cap\Om_\de} U^*_\ve\, dS_x+o(e^{-\frac{\de'}{\ve}}) \ \mbox{ as } \ \ve\to 0^+,
$$
for any $\de'<\de$. 
\par
Next, let $z\notin\Om$ such that $\pa B_{r_e}(z)\cap \pa\Om=\{ q\}$ and choose $\rho>r_e$ such that $\ol{B_\rho(z)}\cap (\pa\Om\setminus B_\de(q))=\varnothing$. We can assume that $z=0$.
Also, let $\Upsilon_\ve$ be given by \eqref{Phi-eps} with $r=r_e$ and $\rho>r_e$.
We check that
$$
\Upsilon_\ve\ge 0=U^*_\ve \ \mbox{ on } \ \pa\Om\cap B_\rho(z), \quad 
\Upsilon_\ve=1>U^*_\ve \ \mbox{ on } \ \Om\cap \pa B_\rho(z).
$$ 
Thus, by comparison, we infer that 
$$
U^*_\ve(x)\le \Upsilon_\ve \ \mbox{ on } \ \ol{B_\rho(z)\cap\Om}.
$$
\par
Now, we can choose $\de>0$ such that
$$
\pa B_R(p)\cap\Om_\de\subset \pa B_R(p)\cap B_\rho(z).
$$
Hence, we deduce that
$$
\int_{\pa B_R(p)} U^*_\ve\, dS_x\le \int_{\pa B_R(p)\cap\Om_\de} \Upsilon_\ve(x-z)\, dS_x+o(e^{-\frac{\de'}{\ve}}).
$$
\par
Moreover, since
$$
\fhi_\ve(x)=e^{-\frac{\rho-|x|}{\ve}} \bigl[1+o(1)\bigr] \ \mbox{ and } \ \psi_\ve(x)=e^{-\frac{|x|-r_e}{\ve}} \bigl[1+o(1)\bigr] \ \mbox{ as } \ve\to 0^+,
$$
for $r_e\le |x| \le (r_e+\rho)/2$, from \eqref{Phi-eps} we deduce that
$$
\Upsilon_\ve(x)=O(e^{-\frac{\rho-r_e}{\ve}}) \ \mbox{ as } \ve\to 0^+,
$$
which easily gives:
$$
\int_{\pa B_R(p)\cap\Om_\de}\Upsilon_\ve(x-z)\,dS_x=O(e^{-\frac{\rho-r_e}{\ve}}) \ \mbox{ as } \ve\to 0^+.
$$
\par
The desired claim then easily follows.
\end{proof}

\begin{proof}[Proof of Theorem \ref{thm:mean-value-asymptotics}]
We first prove formula \eqref{MS-f-non-constant}. Formula \eqref{MS-f-non-constant-heat} will then ensue, thanks to the application of a {lemma for the Laplace transform}.
\par
(i) Fix $\de>0$ and set:
$$
m_\de=\inf_{B_\de(q)\cap\pa\Om} f, \quad M_\de=\sup_{B_\de(q)\cap\pa\Om} f.
$$
Clearly, $m_\de\le M_\de$ and $m_\de, M_\de\to 0$, as $\de\to 0$.
\par
Next, we define on $\pa\Om$ the functions:
$$
f_\de^-=m_\de-(\ul{f}-m_\de)\,\one_{\pa\Om\setminus B_\de(q)}, \quad
f_\de^+=M_\de+(\ol{f}-M_\de)\,\one_{\pa\Om\setminus B_\de(q)}.
$$
If we let $W_\ve$ be the solution of \eqref{heat-resolvent-dirichlet} in $\Om$ with $f=1$ and $U^*_\ve$ the function defined in Lemma \ref{lem:U-eps-*}, by comparison we infer that
$$
m_\de\, W_\ve-(\ul{f}-m_\de)\, U^*_\ve \le U_\ve\le M_\de\, W_\ve+(\ol{f}-M_\de)\, U^*_\ve
\ \mbox{ in } \ \Om.
$$
\par
Now, we know from \cite[Theorem 2.3]{MS1} that
$$
\lim_{\ve\to 0^+}(2\pi \ve)^{-\frac{N-1}{2}}\int_{\pa B_R(p)} W_\ve(x)\,dS_x=\left\{ \prod\limits_{i=1}^{N-1}\bigl[1/R-\ka_i(q)\bigr]\right\}^{-\frac12}.
$$
Thus, thanks to Lemma \ref{lem:U-eps-*}, we deduce that
\begin{multline*}
m_\de \biggl\{\prod\limits_{i=1}^{N-1}\bigl[1/R-\ka_i(q)\bigr]\biggr\}^{-\frac12}\le 
\liminf_{\ve\to 0^+}\biggl[ (2\pi \ve)^{-\frac{N-1}{2}}\int_{\pa B_R(p)} U_\ve(x)\,dS_x\biggr]\le \\
\limsup_{\ve\to 0^+} \biggl[(2\pi \ve)^{-\frac{N-1}{2}}\int_{\pa B_R(p)} U_\ve(x)\,dS_x\biggr]\le 
M_\de \biggl\{\prod\limits_{i=1}^{N-1}\bigl[1/R-\ka_i(q)\bigr]\biggr\}^{-\frac12}.
\end{multline*}
The desired conclusion is then obtained by letting $\de$ tend to $0$.
\par
(ii) Set
$$
F(t)=\int_{\pa B_R(p)} u(x,t)\,dS_x \ \mbox{ for } \ t>0.
$$ 
Thanks to \eqref{laplace},  we have that
$$
\int_{\pa B_R(p)} U_{\sqrt{t}}(x)\,dS_x=\int_0^\infty e^{-\frac{\tau}{t}} F'(\tau)\,d\tau,
$$
after an application of Fubini's theorem and an integration by parts.
We know from formula \eqref{MS-f-non-constant} that 
$$
\lim_{\ve\to 0^+}(4\pi^2 t)^{-\frac{N-1}{4}}\int_{\pa B_R(p)} U_{\sqrt{t}}(x)\,dS_x=\frac{f(q)}{ \sqrt{\prod\limits_{i=1}^{N-1}\bigl[1/R-\ka_i(q)\bigr]}}.
$$
\par
Therefore, {\eqref{MS-f-non-constant-heat} follows from Lemma \ref{lem:feller} if $f(q)>0$. 

Finally, if $f=0$, \eqref{MS-f-non-constant-heat} is a direct consequence of \eqref{MS-f-non-constant}, \eqref{laplace} and the monotonicity in time of $u$}.
\end{proof}

\subsection{The case of sign-changing Dirichlet data}
In this section, we allow the boundary data $f$ to change sign. In this case, the asymptotics depends on how close the point $x$ is to the sets where $f$ is either positive or negative.
As stated in the introduction, a straightforward consequence of Theorem \ref{thm:varadhan-elliptic-parabolic} is Corollary \ref{cor:sign-changing}.

\begin{proof}[Proof of Corollary \ref{cor:sign-changing}]
We give the proof of the elliptic case. The proof of the parabolic asymptotics can be derived either similarly or, as done before, by using Lemma \ref{lem:laplace-transform}.
\par
We examine the case in which $d_{\cS^+}(x)<d_{\cS^-}(x)$. The other case can be recovered by switching the roles of $d_{\cS^+}(x)$ and $d_{\cS^-}(x)$.
\par
By the linearity of the problem \eqref{resolvent-dirichlet}, we have that
$$
U_\ve=U_\ve^+-U_\ve^-,
$$
where $U_\ve^\pm$ are the solutions of \eqref{resolvent-dirichlet} with $f=f^\pm$, where $f^\pm$  are the positive and negative parts of $f$.
\par
Theorem \ref{thm:varadhan-any-positive} gives that
$$
\lim_{\ve\to 0^+} \left[-\ve \log \frac{U_\ve^-(x)}{U_\ve^+(x)}\right]=d_{\cS^-}(x)-d_{\cS^+}(x),
$$
for $x\in\Om\cup\cS^\pm$.
Hence, we infer that
$$
\lim_{\ve\to 0^+} \frac{U_\ve^-(x)}{U_\ve^+(x)}=0 \ \mbox{ for } \ d_{\cS^+}(x)<d_{\cS^-}(x).
$$
Thus, we deduce that
$$
-\ve \log |U_\ve(x)|=-\ve \log U_\ve^+(x)-\ve\log\left|1+\frac{U_\ve^-(x)}{U_\ve^+(x)}\right|=
d_{\cS^+}(x)+o(\ve) \ \mbox{ as } \ \ve\to 0^+,
$$
which gives the desired formula. 
\end{proof}

When $d_{\cS^+}(x)=d_{\cS^-}(x)$, the asymptotics is more delicate, since the vanishing behavior of $f$ at the boundary of its support comes into play. The following remark shows that things may not go as expected.

\begin{rem}[On the threshold case]
\label{rem:threshold}
{\rm
(i) Let $\Om$ be a ball centered at the origin and consider the problem \eqref{heat-resolvent-dirichlet}. 
\par
Choose a continuous function $f^+$ on $\pa\Om$ such that the closure of its positivity set $\cS^+$ is contained in the upper open hemisphere. Then, define the function $f^-$ by
$$
f^-(x', x_N)=f^+(x', -x_N) \ \mbox{ for} \ (x', x_N)\in\pa\Om.
$$
\par
By symmetry, it is clear that the solution $U_\ve$ of problem \eqref{heat-resolvent-dirichlet}, with Dirichlet boundary values $f=f^+-f^-$ is zero on the set $\cE=\{ x\in\Om: x_N=0\}$ --- the equatorial cross section of $\Om$. This means that $-\ve \log U_\ve=+\infty$  for any $\ve>0$, and hence
$$
\lim_{\ve\to 0^+} \bigl[-\ve \log U_\ve\bigr]=+\infty
$$
on $\cE$, where clearly $d_{\cS^+}=d_{\cS^-}$. 
\par
Of course, we can imagine more complicated situations, in which the function $f$ is also (oddly) symmetric  with respect to other hyperplanes passing through the origin. We would then obtain that $-\ve \log U_\ve=+\infty$ on those hyperplanes.
\par
An analogous example is the following.
For $k=1,2,\dots$, let $P_{N,k}$ be a $k$-homogeneous harmonic polynomial and $Y_{N,k}$ the corresponding spherical harmonic. Then the solution of  \eqref{heat-resolvent-dirichlet} in $\Om= B_1(0)$, with $f = P_{N,k} = Y_{N,k}$ on $\partial \Omega$, is given by 
\begin{equation*}
\label{e:separation_solution}
    U_\ve(x) = \frac{\displaystyle \int_0^\pi e^{|x|\cos\theta/\ve}(\sin\theta)^{2k+N-2}\,d\theta}{\displaystyle \int_0^\pi e^{\cos\theta/\ve}(\sin\theta)^{2k+N-2}\,d\theta}\, P_{N,k}(x).
\end{equation*}
\par
(ii)
If we examine item (i) more deeply, we also
notice that, in the subdomains $\Om^+=\{ x\in\Om : x_N>0\}$ and $\Om^-=\{ x\in\Om : x_N<0\}$, the solution $U_\ve$ in $\Om$ with the prescribed symmetric data concides in each set $\Om^+$ and $\Om^-$ respectively  with the solutions $U_\ve^\pm$ of the problems
$$
\ve^2 \De U^\pm_\ve-U^\pm_\ve=0 \ \mbox{ in } \ \Om^\pm, \quad U^\pm_\ve=f^\pm_* \ \mbox{ on } \ \pa\Om^\pm,
$$
where 
$$
f^\pm_*=\begin{cases}
f^\pm \ &\mbox{ on } \ \pa\Om^\pm\setminus\cE, \\
0 \ &\mbox{ on } \ \cE.
\end{cases}
$$
\par
Thus, Theorem \ref{thm:varadhan-elliptic-parabolic} applies to $U^\pm_\ve$ (with some slight modification due to the corners), and hence $-\ve \log U_\ve^\pm$ converges to $d_{\cS^\pm}$ as $\ve\to 0^+$, uniformly on compact subsets of $\Om^\pm\cup\cS^\pm$. In other words, we 
can say that $-\ve \log |U_\ve|$ converges uniformly on compact subsets of 
$$
\{ x\in\Om\cup\cS: d_{\cS^+}(x)\ne d_{\cS^+}(x)\}
$$
and its limit $V$ can be extended continuously to the whole $\ol{\Om}$. In particular, $V=d_\cS$ on $\cE$.
\par
(iii) It remains open the problem of providing sufficient conditions on the data, which give pointwise or uniform convergence at points $x\in\Om$ such that $d_{\cS^+}(x)=d_{\cS^-}(x)$.
}
\end{rem}

\section{Time invariant and other invariant surfaces}
\label{sec:invariant-surfaces}
In this section, we shall consider the solution 
$u=u(x,t)$ of the general Cauchy-Dirichlet problem \eqref{cauchy-dirichlet-gen}:
\begin{equation*}
u_t=\De u \ \mbox{ in } \ \Om\times(0,\infty), \ u=u_0  \ \mbox{ on } \ \Om\times\{ 0\}, \
 u=f \ \mbox{ on } \ \pa\Om\times(0,\infty),
\end{equation*}
where $u_0\in L^2(\Om)$ and $f\in C(\pa\Om)$.
\par
We will investigate on the possible shape of time-invariant surfaces for $u$. As defined in the introduction, a (connected) surface $\Ga\subset\Om$ of codimension $1$ is called a \textit{time-invariant surface}  for $u$ if 
\begin{equation*}
u(x,t)=c(t) \ \mbox{ for } \ (x,t)\in \Ga\times (a,b),
\end{equation*}
for some function $c:(a,b)\to\RR$ of the time $t$. 

 The relevant regularity of $\Ga$ will be specified when needed. 

\par
We start by examining the case of time-invariant surfaces for large times.

\subsection{Invariant surfaces for large times}
If $\pa\Om$ is made of regular points, the Dirichlet problem
\begin{equation}
\label{harmonic}
\De h=0 \ \mbox{ in } \ \Om, \quad h=f  \ \mbox{ on } \ \Om,
\end{equation}
admits a unique solution of class $C(\ol{\Om})\cap C^2(\Om)$.
Thus, the function $\psi=u-h$ is the solution of \eqref{cauchy-dirichlet-gen}, with $f$ replaced by $0$ and $u_0$ replaced by $u_0-h$. Note that $u_0-h$ still belongs to $L^2(\Om)$.
\par
The spectral theory for the heat equation shows that $u-h$ can be expanded as a series of Dirichlet eigenfunctions of the Laplace operator (see \cite{CH}). In particular,
it holds that
\begin{equation}
\label{infinity}
\lim_{t\to \infty} e^{\la_1 t} \psi(x,t)= \left[\int_\Om (u_0-h)\,\phi_1\,dx\right]\phi_1(x),
\quad x\in\ol{\Om},
\end{equation}
where $\la_1>0$ is the first Dirichlet eigenvalue of the Laplacian
and $\phi_1$ is the corresponding eigenfunction, normalized in $L^2(\Om)$. As a matter of fact, $\phi_1$ satisfies the problem
\begin{equation*}
\De \phi_1 +\la_1 \phi_1 =0 \ \mbox{ and } \ \phi_1>0 \ \mbox{ in } \ \Om,
\quad \phi_1=0 \ \mbox{ on } \  \pa\Om,
\end{equation*}
and is such that $\int_\Om \phi_1^2\, dx=1$.

The convergence in (\ref{infinity}) is uniform on
$\ol{\Om}.$ Indeed, by letting $\Psi(x,t) = e^{\la_1 t}\psi(x,t)$,
since the convergence in \eqref{infinity} is $L^2(\Om)$ as $t \to \infty$ (see \cite{CH} or \cite{Ga}), we have that $\nr
\Psi(\cdot,t)\nr_{L^2(\Om)}$ is uniformly bounded for $t\in [0,\infty)$.
Hence, by an a priori bound of Alikakos (see \cite{Ali}, pp. 207 -- 211 or
\cite{He}, p. 74), $\nr \Psi(\cdot,t)\nr_{L^\infty(\Om)}$ is also uniformly
bounded for $t\in [0,\infty)$. Furthermore, by the standard parabolic
estimates (see \cite{LSU}), there exists a number $\alpha \in (0,1)$ such
that the H\"older norm $\nr\Psi(\cdot,t)\nr_{C^\alpha(\overline{\Omega})}$ is
uniformly bounded
for $t\in [1,\infty)$. This ensures the uniform convergence.
In particular, as $t\to\infty$, $\psi$ converges exponentially to zero, uniformly on $\ol{\Om}$.
\par
All in all, we conclude that
\begin{equation}
\label{spectral-asymptotics}
u(x,t)=h(x)+\left[\int_\Om (u_0-h)\,\phi_1\,dx\right] \phi_1(x)\,e^{-\la_1 t}+o(e^{-\la_1 t}) \ \mbox{ as } \ t\to \infty.
\end{equation}
uniformly on $\ol{\Om}$.
\par
In view of these considerations, the proof of following result is then straightforward. 

\begin{thm}
\label{thm:large-time}
Let $\pa\Om$ be made of regular points for the Dirichlet problem. Let $u=u(x,t)$ be the solution of problem \eqref{cauchy-dirichlet-gen}. Suppose that, for a surface $\Ga\subset\Om$ of $u$, \eqref{time-invariant} holds with $b=\infty$.
Then, $\Ga$ is a level surface of the solution $h$ of \eqref{harmonic} and, if
$$
\int_\Om (u_0-h)\,\phi_1\,dx\ne 0,
$$
of the first eigenfunction $\phi_1$.
\end{thm}


As a consequence, we obtain a first simple rigidity result.
\begin{cor}
\label{cor:large-time}
If $\Ga=\pa D$, with $\ol{D}\subset\Om$, is a time-invariant surface for the solution of problem \eqref{cauchy-dirichlet-gen} with $b=\infty$, then $f$ must be constant. 
\end{cor}

\begin{proof}
Under this assumption, since $\Ga$ is a level surface of $h$, we have that $h$ is identically constant on $\ol{D}$ and, by analytic continuation, on $\ol{\Om}$. Thus, $f$ must be constant.
\end{proof}

\begin{rem}
{\rm
The arguments presented in Theorem \ref{thm:large-time} and Corollary \ref{cor:large-time} can be repeated for the solution of \eqref{cauchy-dirichlet}. In order to do that, all what we need is to assume on $A(x)$ sufficient conditions (satisfied by our assumptions), which make sure that the following properties are satisfied.
\begin{enumerate}[(i)]
\item
The Dirichlet problem 
\begin{equation*}
\tr \bigl[A(x)\,\na^2 h\bigr]=0 \ \mbox{ in } \ \Om,  \quad
h=f \ \mbox{ on } \ \pa\Om,
\end{equation*}
admits a solution $h$ of class, say, $C(\ol{\Om})\cap C^2(\Om)$ (see \cite{LSW, Ol}).
\item
A spectral asymptotic formula of the kind \eqref{spectral-asymptotics} holds true. Here,  $\phi_1$ is a suitably normalized non-trivial solution of the problem
$$
\tr \bigl[A(x)\,\na^2 \phi_1\bigr]+\la_1^A\,\phi_1=0 \ \mbox{ and } \ \phi_1>0 \ \mbox{ in } \ \Om, \quad \phi_1=0 \ \mbox{ on } \ \pa\Om,
$$
for some $\la_1^A>0$ (see \cite{CH}).
\item
A unique continuation principle for solutions of 
$
\tr \bigl[A(x)\,\na^2 h\bigr]=0
$
holds in $\Om$ (see \cite[Chap. 18]{Ve}).
\end{enumerate}
\par
In particular, if properties (i)-(iii) hold true, we would get that the presence of a time-invariant (for large times) surface $\Ga=\pa D$, with $\ol{D}\subset\Om$, forces $f$ to be constant. 
}
\end{rem}

\begin{rem}
{\rm
When $\Ga$ extends up to $\pa\Om$, the large-time behavior of $u$ only gives the information contained in Theorem \ref{thm:large-time} and, possibly, that $\Ga$ is a level surface of all the Dirichlet eigenfunctions which are not orthogonal to $u_0-h$. Thus, in order to obtain some new interesting rigidity results, in the next section we examine the implications of the short-time behavior of $u$.
}
\end{rem}

\subsection{Invariant surfaces for short times: the case $f\ge 0$}
\label{subsec:short-times}

Note, {from now on, we set $u_0 \equiv 0$.} We begin with the following lemma, which essentially says that the case in which $\Ga=\pa D$, with $\ol{D}\subset\Om$ is quite general.

\begin{lem}
\label{lem:1}
Let $\Om$ be a bounded domain 
and set $a=0$. 
Suppose that 
$f \in C(\pa\Om)$ is non-negative, with $\cS\ne\varnothing$ given by \eqref{def-S}.
Let $u$ be the solution of \eqref{heat-cauchy-dirichlet}. 
\par
If  \eqref{time-invariant} holds  for some $\Ga\subset\Om$, then the following assertions hold true:
\begin{enumerate}[(i)]
\item
$\ol{\Ga}\cap\pa\Om=\varnothing$;
\item
 if $D\subset\Om$ is a domain with connected boundary and  $\Ga=\pa D\cap\Om\ne \varnothing$, then 
$\ol{D}\subset\Om$.
\end{enumerate}
\end{lem}

\begin{proof}
(i) By the strong maximum principle, we infer that $c(t)>0$ for $t>0$. Thus, if by contradiction $\ol{\Ga}$ intersects $\pa\Om$ at some point $x_0$, then $x_0\in \cS$, otherwise, being as $x_0$ a regular point and $f$ continuous, we would get by continuity  the contradiction $0<c(t)=f(x_0)=0$. On the other hand, when $x_0\in\cS$, again by continuity, we infer that $c(t)=f(x_0)>0$ for every $t>0$. This is not possible, because we would have that $0<f(x_0)=c(t)=u(z,t)\to 0$ as $t\to 0^+$, for any $z\in\Ga$.
\par
(ii)
Note that $\pa D$ is the union of the two closed sets $\ol{\Ga}$ and $\ol{D}\cap\pa\Om$. By item (i) we know that these are disjoint. 
Since $\Ga\ne\varnothing$ and $\pa D$ is connected, then necessarily $\ol{D}\cap\pa\Om=\varnothing$. Therefore, $\ol{D}\subset\Om$.
\end{proof}

By merging this proposition with Theorem \ref{thm:large-time}, we discover that the rigidity result obtained in \cite{MS1}, when $f$ is constant on $\pa\Om$, still holds when $f$ is allowed to be non-constant.

\begin{proof}[Proof of Theorem \ref{thm:rigidity-definitive}]
Under the relevant assumptions, we can apply Corollary \ref{cor:large-time} and Lemma \ref{lem:1}. In fact, item (ii) of this lemma gives that 
$\ol{D}\subset\Om$. 
From Corollary \ref{cor:large-time} we then deduce that $f$ must be a non-zero constant, being as $\cS\ne\varnothing$. 
\par
Thus, up to a multiplicative constant, we are in the situation of \cite[Theorem 1.1]{MS1}.
The desired conclusion then follows from that theorem.
\end{proof}

\begin{rem}
{\rm
It should be noticed that, in \cite{MS4}[Theorem 1.2], the conclusion of Theorem \ref{thm:rigidity-definitive} is obtained only as the consequence of the short-time behavior, via the method of moving planes (see also \cite{CMS}). Notice that, when $f$ is constant, Theorem \ref{thm:large-time} does not give any extra useful information. The large-time behavior of the solution is thus crucial when $f$ is not constant.
\par
If we want to rely only on the short-time behavior of the solution, then, in order to obtain interesting rigidity results, we need assume extra sufficient conditions.}
\end{rem}

\par
 In what follows, we present results in the positive and some counterexamples. 
As starters, 
we will prove the two-invariant-surfaces result announced in Theorem \ref{thm:two-surfaces}.
Before proceeding to its proof, we need to adapt \cite[Lemma 3.1]{MS1} to the assumptions in hand. To this aim, we recall the notations stated in Section \ref{subsec:Omega}.

\begin{lem}[Analitycity of invariant surfaces]
\label{lem:lemma3.1ms1_revisited}
Let $\Omega$ be a bounded domain in $\mathbb R^N$.
Assume that either
\begin{enumerate}[(a)]
\item 
     $\Om$ satisfies the exterior sphere condition and $\cS = \pa\Om$, or
     \item $\pa\Om \in C^{1,1}$ and $\ol{\cS}=\pa\Om$.
\end{enumerate}
  \par  
Suppose that  $D$ is a domain with boundary $\Ga\subset\Om$ satisfying the interior cone condition, and
that \eqref{time-invariant} holds with $a=0$ and $b=\infty$.
Then, it holds that
    \begin{enumerate}[(i)]
        \item  
        there exists $R>0$ such that $d_{\partial\Omega}(x)=R$, for every $x \in \Gamma$;
        \item $\Gamma$  is analytic.
    \end{enumerate}
\end{lem}

\begin{proof}
    (i) We recall that, by Proposition \ref{prop:varadhan-easy-extension},  (a) gives that $-4 t \log u(x,t)\to d_{\partial \Om}(x)$, as $t\to 0^+$. The convergence is uniform on $\ol{\Om}$.
We derive the same conclusion also if (b) is in force. In fact, in this case, we have that  $d_\cS=d_{\ol{\cS}}=d_{\pa\Om}$. Hence, we conclude thanks to Theorem \ref{thm:varadhan-elliptic-parabolic}.
However, the convergence holds uniformly on compact subsets of $\Om\cup \cS$, instead. 
\par
In any case,   for every $x \in \Ga$, we infer that 
$$
d_{\pa\Om}(x)=\lim_{t \to 0^+}-4 t \log u(x,t)=\lim_{t\to 0^+}-4t \log c(t).
$$
The last quantity is independent of $x$ and we can set it equal to an $R\ge 0$. Since, $\Ga \subset \Om$, we infer that $R>0$.
\par
(ii) We follow the main steps of the proof of \cite[Lemma 3.1 (ii)]{MS1}. It is convenient to work with the functions $U_\ve$, defined in \eqref{laplace}. Indeed, as already observed, if $\Ga$ is time-invariant for $u(x,t)$ for $t\in (0,\infty)$, then $U_\ve=c_\ve$  on $\Ga$, for some constants $c_\ve$ and any $\ve>0$. 
\par
In order to prove the analyticity of $\Ga$, since any $U_\ve$ is analytic, it is sufficient to show that, for any $x_0\in\Ga$, there is an $\ve_0>0$ such that $\na U_{\ve_0}(x_0)\ne 0$. Hence, if by contradiction $\na U_{\ve}(x_0)=0$ for any $\ve>0$, then, as done in \cite[Lemma 3.1 (ii)]{MS1}, we infer that 
$$
\int_{\pa B_R(x_0)}(x-x_0)\,U_\ve(x)\,dS_x =0 \ \mbox{ for any  $\ve >0$. } 
$$
Thus, in order to get a contradiction, up to translating $x_0$ to $0$ and rotating the axes, it suffices to show that there exists $\ve^*>0$ such that
$$
\int_{\pa B_R}  x_N\,U_\ve(x)\,dS_x >0 \ \mbox{ for $0<\ve<\ve^*$.} 
$$
Notice that, in these coordinates, the truncated cone $K\subset D$ touching $x_0=0$ can be parametrized as $K=\{ x\in B_\rho: x_N<-|x|\,\cos\te\}$ for some $\rho\in (0,R)$ and $\te\in (0, \pi/2)$. Also, the set $V=\{  x\in\pa B_R: x_N>R\,\sin\te\}$ contains $\pa\Om\cap\pa B_R$. 
\par
The proof of the case in which $f>0$ runs exactly as done in the proof of \cite[Lemma 3.1 (ii)]{MS1}. Indeed, by the uniform convergence of $-\ve \log U_\ve$ to $d_{\pa\Om}$ on $\ol{\Om}$, we can choose $\de>0$ and $\ve^*>0$ such that
\begin{multline}
\label{ms-am2002}
\int_{\pa B_R}  x_N\,U_\ve(x)\,dS_x\ge \\
Re^{-3\de/\ve}\left[\sin\te\, \cH^{N-1} (V\cap \Om_{2\de})-\frac12\,e^{-\de/\ve}\cH^{N-1}(\pa B_R) \right],
\end{multline}
for $0<\ve<\ve^*$. Therefore, the right-hand side can be made positive by choosing a smaller $\ve^*$, if needed.

Now, consider the case $f \ge 0$ with $\overline{S}=\partial\Om$.  All the previous arguments apply once we can employ our version of the Varadhan formula, i.e. Theorem \ref{thm:varadhan-elliptic-parabolic}, that holds under the assumption $\pa\Om\in C^{1,1}$. 
The only difference is due to the fact that now the relevant convergence is uniform  on compact subsets of $\Om\cup\cS$. 
\par
Nevertheless, we observe that \eqref{ms-am2002} can be modified as
\begin{multline*}
\int_{\pa B_R}  x_N\,U_\ve(x)\,dS_x\ge \\
Re^{-3\de/\ve}\left[\sin\te\, \cH^{N-1} (V\cap\Om_{2\de,\eta})-\frac12\,e^{-\de/\ve}\cH^{N-1}(\pa B_R) \right],
\end{multline*}
where $\Om_{2\de,\eta}=\{x \in \Om: \eta < d_{\pa\Om}(x)<2\de\}$ and $\eta>0$ is sufficiently small. This is possible because $-\ve\log U_\ve$ converges to $d_{\partial\Omega}$, uniformly on $\ol{\Om}_{2\de,\eta}$, and hence the same computations used in \cite{MS1} still work.
\par
Therefore, as in the previous case, we infer that the right-hand side of the last inequality is positive for any sufficiently small $\ve>0$.
\end{proof}

Once we have established the analyticity of $\Ga$, the same (geometric) arguments used in the proof of \cite[Lemma 3.1 (iii)--(vi)]{MS1} apply and give the following corollary.

\begin{cor}
\label{cor:geometric_consequences}
    Under the same assumptions of Lemma \ref{lem:lemma3.1ms1_revisited}, we have:
    \begin{enumerate}[(i)]
        \item $\pa\Om$ is analytic and $\pa\Om=\{ x\in\Om : d_\Ga(x)=R\}$;
        \item the mapping $\Ga\ni x \mapsto y(x)= x-R\,\nu^\Ga(x)\in \pa\Om$ is a diffeomorphism; 
        \item for every $x \in \Ga$, $\na d_{\pa\Om}(x)=\nu^\Ga(x)$ and $\pa\Om \cap \pa B_R(x)=\{y(x)\}$;
        \item for every $y \in \pa\Om$,  $\ka_j(y)<1/R$, for every $j=1,\dots,N-1$. 
\end{enumerate}
Here, $\nu^\Ga$ is the interior unit normal to $\Ga$.
\end{cor}

We now have all the ingredients to prove our rigidity result.

\begin{proof}[Proof of Theorem \ref{thm:two-surfaces}]
First, we recall that $U_\ve$ satisfies a mean value property for $U_\ve$. In fact, it holds that
\begin{equation}
\label{e:mean_value_ve}
U_\ve(p)=\frac{\psi_N(r/\ve)}{|\partial B_r|}\int_{\partial B_r(p)}U_\ve(x)\,dS_x, 
\end{equation}
for every $\ol{B_r(p)}\subset\Om$,  where 
$$
\psi_N(\si)=\frac{\int_0^\pi (\sin \te)^{N-2}\,d\te}{\int_0^\pi e^{\si\cos\te}(\sin\te)^{N-2}\,d\te} \ \mbox{ for } \ \si > 0.
$$
(See e.g. \cite{HMOSY, Ku}.)
\par
As already observed, \eqref{time-invariant} gives that that $U_\ve=c_i^\ve$ on $\Ga_i$ for some positive constants $c_i^\ve>0$.
    Since the right-hand side in \eqref{e:mean_value_ve} is continuous as $r \to d_{\pa\Om}(p)$, then by passing to the limit and by using that Lemma \ref{lem:lemma3.1ms1_revisited} (i), we get that
\begin{equation*}
c_\ve^i = U_\ve(p)=\frac{\psi_N(R_i/\ve)}{|\pa B_{R_i}| }\int_{\pa B_{R_i}(p)}U_\ve(x)\,dS_x, \ p\in\Ga_i,
\end{equation*}
for each $i=1,2$.
\par
Corollary \ref{cor:geometric_consequences} makes sure that the assumptions of Theorem \ref{thm:mean-value-asymptotics} are satisfied for both $\Ga_1$ and $\Ga_2$. Hence, passing to the limit as $\ve\to 0^+$ gives that
\begin{equation}
\label{eq:2_surfaces}
f(q)\prod_{j=1}^{N-1}\left[1/R_i-\kappa_j(q)\right]^{-1/2}=c_i \ \mbox{ for any } \ q\in\pa\Om, \ i=1,2,
\end{equation}
for some constants $c_i=c_i^0$. 
We note that $f>0$ on $\pa\Om$, because \eqref{eq:2_surfaces} rules out that $f$ vanishes somewhere.
Thus, we deduce that 
$$
\prod_{j=1}^{N-1}\frac{1-R_1\,\ka_j(q)}{1-R_2\,\ka_j(q)}= \frac{c_1}{c_2} \ \mbox{ for any } \ q\in \pa\Om.
$$
\par
Now, without loss of generality, we assume $R_2>R_1$ and consider the function defined by
 $$
 \Psi(s_1,\cdots,s_{N-1})=\prod_{j=1}^{N-1} \frac{1-R_1\,s_j}{1-R_2\,s_j},
 $$
 for $s_1 \ge \cdots, s_{N-1}$. We check that
$$
\frac{\pa_{s_k}\Psi(s_1,\cdots,s_{N-1})}{\Psi(s_1,\cdots,s_{N-1})}=\pa_{s_k}\log \Psi(s_1,\cdots,s_{N-1})=\frac{R_2-R_1}{(1-R_1\,s_k)(1-R_2\,s_k)}, 
$$
and hence, $\pa_{s_k}\Psi(s_1,\cdots,s_{N-1})>0$ for $k=1,\dots, N-1$.
\par
Therefore, an application of Theorem \ref{thm:Al} gives that $\pa\Om$ must be a sphere of radius $R$ such that 
$$
\prod_{j=1}^{N-1}\frac{1-R_1 R}{1-R_2 R}= \frac{c_1}{c_2}.
$$
Since $\Ga_1$ and $\Ga_2$ are parallel to $\pa\Om$, they must also be spheres, concentric with $\pa\Om$. 
Finally, \eqref{eq:2_surfaces} gives that $f$ must be constant. 
\end{proof}

\subsection{Invariant surfaces for short times: the case of sign-changing $f$}
\label{subsec:short-times-f-sign-changing}
In this section, we will prove Theorem \ref{thm:non-existence}.
Since $f$ is allowed to change sign, we use the notation of Section \ref{subsec:def-S} for the sets $\cS^\pm$, $\cS$, $\Om^\pm$, $G$, $G_\reg$, and $G_\sing$. We begin with a preliminary result.

\begin{lem}[The location of $\Ga$]
\label{lem:location-invariant-surface}
Let $\Om\subset\RR^N$, $N\ge 2$, be a bounded domain with boundary $\pa\Om$ of class $C^{1,1}$. Let $f\in C(\pa\Om)$, suppose that it changes sign, i.e.  the sets $\cS^\pm$ are non-empty, and assume that $G$ has no interior points. 
Let $u=u(x,t)$ be the solution of \eqref{heat-cauchy-dirichlet}.
\par
If $\Ga=\pa D$, with $\ol{D}\subset\Om$ and \eqref{time-invariant} holds  for $a=0$, then 
\begin{enumerate}[(i)]
\item  $\Gamma\not\subseteq G$\item either $D\subset \overline{\Om^+}$ or $D\subset \overline{\Om^-}$;
\item  there exists $R>0$ such that $d_{\mathcal S}=R$ on $\overline{\Gamma \setminus G}$.

\end{enumerate}
\end{lem}

\begin{proof} It is clear that $D$ is not contained in $G$, since $G$ has empty interior.
Without loss of generality, we can then assume that $D\cap\Om^+$ contains a point $x$, namely, $d_{\cS^+}(x)<d_{\cS^-}(x)$. Next, let $y\in\ol{\cS^+}$ be such that $|x-y|=d_{\cS^+}(x)$.
Since $\ol{D}\subset\Om$, on the segment joining $x$ to $y$, there is a point $x'$ belonging to $\Ga$. It is immediate to check that 
$$
d_{\cS^+}(x')=d_{\cS^+}(x)-|x-x'|<d_{\cS^-}(x')-|x-x'|\le d_{\cS^-}(x).
$$ 
The last inequality follows from the triangle inequality. We thus infer, in particular, that $x\in\Ga\setminus G$, and hence, (i) is proved.
\par
Next, by the same arguments, we have that $D\cap\Om^\pm\ne\varnothing$ if and only if $\Ga\cap\Om^\pm\ne\varnothing$. Now, by Corollary \ref{cor:sign-changing}, 
we have that
$u(x, t) > 0$ for $x\in\Om^+$ and $u(x, t) < 0$ for $x\in\Om^-$, if $t$ is sufficiently small.
Thus, $\Ga=\pa D$ cannot
intersect both $\Om^+$ and $\Om^-$. Item (ii) then follows at once.
\par
Lastly, the conclusion of item (iii) follows from the Varadhan formula (Corollary
\ref{cor:sign-changing}), the fact that $d(\Ga, \cS)>0$, and the continuity of $d_\cS$.
\end{proof}

\vskip.2cm

\begin{proof}[Proof of Theorem \ref{thm:non-existence}]
(i) First, we prove that $\Ga$ is analytic. In fact, as observed, the modified Laplace transform
of $u$, $U_\ve$, defined by \eqref{laplace}, satisfies \eqref{heat-resolvent-dirichlet}.
Since $U_\ve=c_\ve$ on $\pa D$, then, depending on the sign of $c_\ve$, by the maximum principle applied in $D$, we see that either the maximum or the
minimum of $U_\ve$ is attained at the boundary points.
Thus, the strong maximum principle and the Hopf lemma give that $\na U_\ve\ne 0$ on $\Ga$. Thence,  $\Ga$ is analytic, by the implicit function theorem, being as $U_\ve$ analytic in $\Om$.
\par
(ii) Next, we prove that any $x\in G_\reg\cap\Om$ lies on a segment contained in $G_\reg$.
In fact, by the definition of $G_\reg$, there is a unique point $y\in\pa\Om$ such that $d_{\pa\Om}(x)=|x-y|$. Since $x\in G$, the uniqueness gives that $y\in\ol{\cS^+}\cap\ol{\cS^-}$. It is easy to show that $y$ also (uniquely) minimizes its distance to any point of the segment $\si$ joining $x$ to $y$. Consequently, $\si$ is contained in $G_\reg$.
\par
(iii) We have that $d_\cS=R$ on $\Ga$, for some $R>0$. Indeed, thanks to item (ii) of Lemma \ref{lem:location-invariant-surface}, all we have to do is to show that $\Ga\setminus G$ is dense in $\Ga$. 
\par
By contradiction, let $B_r(x)$ be a small ball centerd at $x\in\Ga$, which does not intersect $\Ga\setminus G$. We thus have that $B_r(x)$ does intersect $\Ga\setminus G_\sing$, since $\cH^{N-1}(G_\sing)=0$. In particular, the set $B_r(x) \cap \Ga\cap G_\reg$ contains a point $x'$. Since $G_\reg$ is relatively open, there exists a smaller radius $r'$ such that $B_{r'}(x') \cap G\subset G_\reg$ and $B_{r'}(x')\subset B_{r}(x)$. Up to further decreasing $r'$, we can also assume that $B_{r'}(x') \cap G$ is contained in one (say $G_i$) of the $k$ regular surfaces constituting $G_\reg$. Being as $B_{r'}(x')\subset B_{r}(x)$, we infer that $B_{r'}(x')$ does not intersect $\Ga\setminus G_i$, and hence $B_{r'}(x')\cap\Ga\subset B_{r'}(x')\cap G_i$.
Consequently, since the two sets in this inclusion are connected surfaces of class $C^1$ of the same dimension, we conclude that $\Ga$ and $G_i$ coincide in $B_{r'}(x')$.
\par
Now, item (ii) ensures that $\Ga$ contains a segment $\si$, as well. Thus, by the analytycity of $\Ga$, $\Ga$ also contains the segment extending $\si$ up to the boundary $\pa\Om$.
This contradicts that $\ol{D}\subset\Om$.
\par
(iv) Conclusion. Item (iii) and the assumption $\ol{\cS}=\pa\Om$ make sure that $d_{\pa\Om}=R$ on $\Ga$. We then can apply items (ii) and (iii) of Corollary \ref{cor:geometric_consequences} and infer that the mapping $\Ga\ni x\mapsto x-R\,\nu^\Ga\in\pa\Om$ defines a diffeomorphism between $\Ga$ and $\pa\Om$.
\par
Now, since $G$ has empty interior, thanks to item (i) of Lemma \ref{lem:location-invariant-surface},  either $D\subset \overline{\Om^+}$ or $D\subset \overline{\Om^-}$. Assume, for instance, that $D\subset \overline{\Om^+}$. Then we have in particular that $d_\cS=d_{\cS^+}$ on $\Ga$, and hence 
$x-R\,\nu^\Ga\in\ol{\cS^+}$ for $x\in\Ga$. Since this mapping is suriective onto $\pa\Om$, we conclude that $\ol{\cS^+}=\pa\Om$, and hence $\cS^-=\varnothing$. This is a contradiction, because $f$ changes sign.
\end{proof}

We conclude this section with some remarks about the problem of invariant surfaces in the sign-changing regime.
\begin{rem}
\label{rem:IS-sign-changing}
{\rm
(i) We first notice that the setting of problem \eqref{cauchy-dirichlet-gen} is too general, if we want to derive reasonable rigidity results induced by time-invariant surfaces for short times. Indeed, any harmonic function $h$ in $\Om$ is a solution of \eqref{cauchy-dirichlet-gen}, with $u_0=h$ and $f=h_{|\pa\Om}$. In this case, all the level surfaces of $h$ are invariant.  
Moreover, we can construct similar examples, owing to the charachterization contained \cite{MM2}. Thus, the assumption of  problem \eqref{heat-cauchy-dirichlet} (i.e $u_0\equiv 0$ in $\Om$) is important.
\par
(ii) If we consider problem \eqref{heat-cauchy-dirichlet}, instead, we observe that the zero-sets of the solutions constructed in item (i) of Remark \ref{rem:threshold}
are diffusion-invariant sets (and hence time-invariant sets for a suitable parabolic Cauchy-Dirichlet setting) for the solutions of \eqref{heat-resolvent-dirichlet}. Thus, those examples show that invariant ``surfaces'' may have self-intersections and extend up to the boundary, when the Dirichlet boundary values change sign. 
}
\end{rem}

\appendix

\section{Auxiliary results}
\label{appA}
In this appendix, we collect some results already present in the literature, which are instrumental to the proofs of Theorems \ref{thm:varadhan-elliptic-parabolic}, \ref{thm:mean-value-asymptotics}, and \ref{thm:two-surfaces}.
\par
We begin with two lemmas regarding the Laplace transform of a function.

\begin{lem}[\cite{Va}, Theorem 4.7]
\label{lem:laplace-transform}
Let $F_\be:[0,\infty)\to [0,1]$ be non-decreasing functions such that $F_\be(0)=0$, for $\be$ in some set of indices. Consider the Laplace transforms:
$$
G_\be(\tau)=\int_0^\infty e^{-\tau t} dF_\be(t), \ \tau\ge 0.
$$
\par
If
$$
\lim_{\tau\to\infty}\left[-\frac1{\sqrt{2 \tau}} \log G_\be(\tau) \right]=l_\be
$$
uniformly in $\be$ and if $0\le l_\be\le L<\infty$, then
$$
\lim_{t\to 0}\left[-2t \log F_\be(t) \right]=l_\be^2,
$$
uniformly in $\be$.
\end{lem}

The next lemma can be derived from \cite{Fe}.

\begin{lem}[{\cite[Ch.~XIII]{Fe}}]
\label{lem:feller}
Let $F:[0,\infty)\to[0,\infty)$ be a non-decreasing function with
$F(0)=0$. Consider its Laplace transform
\[
G(\tau):=\int_0^\infty e^{-\tau t}\,dF(t),
\qquad \tau>0.
\]
Assume that there exists $\rho\ge0$ such that
\[
\lim_{\tau\to\infty}\frac{G(\lambda\tau)}{G(\tau)}
=\lambda^{-\rho},
\qquad\text{for every }\lambda>0.
\]
Then
\[
\lim_{t\to0^+}\frac{F(tx)}{G(1/t)}
=
\frac{x^\rho}{\Gamma(\rho+1)},
\qquad\text{for every }x>0.
\]
In particular,
\[
\frac{F(t)}{G(1/t)}\to \Gamma(\rho+1)
\qquad\text{as }t\to0^+.
\]
\end{lem}

\begin{proof}
The statement follows from Theorem~1 of Chapter XIII, \S5 in
\cite{Fe}, together with Theorem~3 therein, which asserts that the
result remains valid after interchanging the roles of the origin and
infinity.
\end{proof}

Finally, we recall a version of the celebrated \textit{Soap Bubble Theorem}, which applies to Wirtinger-type surfaces, for which A.~D.~Aleksandrov created his \textit{reflectin principle}.

\begin{thm}[\cite{Ale, Re}]
\label{thm:Al}
    Let $\Psi=\Psi(s_1,\cdots, s_{N-1})$ be a continuously differentiable function, defined for $s_1\ge \cdots\ge s_{N-1}$, and subject to the conditions $\partial_{s_j}\Psi>0$, for $j=1,\dots, N-1$.    Suppose that $\Ga$ is a twice-differentiable closed surface without self-intersections and with bounded principal curvatures $\ka_1, \dots, \ka_{N-1}$.
\par   
If $\Psi(\ka_1,\cdots,\ka_{N-1})$ is constant on $\Ga$, then $\Ga$ is a sphere.
\end{thm}

\section*{Acknowledgements}
The authors are supported by  the Gruppo Nazionale per l'Analisi Matematica, la Probabilit\`a e Applicazioni
(GNAMPA) of the Istituto Nazionale di Alta Matematica (INdAM).
D.B. was also supported by MUR - M4C2 1.5 of PNRR with grant no. ECS00000036. M. M. was also supported by MIMIT - project "SMART E-DRIVING - SED'" n. F/340045/04/X59 - CUP: B89J24001370005.

\end{document}